\documentclass[10pt]{article}

\usepackage{amsmath} 
\usepackage{mathrsfs} 
\usepackage[utf8]{inputenc} 
\usepackage{tikz-cd} 
\usepackage{enumitem} 
\usepackage{multirow} 

\tikzstyle{arrow} = [thick,->]

\usepackage[margin=1in]{geometry}

\usepackage{titlesec}
\titleformat{\section}[hang]{\normalfont\Large\bfseries}{\thesection}{1em}{}
\titleformat{\subsection}[hang]{\normalfont\large}{\thesubsection}{1em}{}
\titleformat{\subsubsection}[hang]{\normalfont}{\thesubsubsection\fontsize{12pt}{12pt}}{1em}{}

\usepackage[citestyle=alphabetic, bibstyle=alphabetic, natbib]{biblatex} 
\usepackage{xcolor} 
\definecolor{DarkGreen}{RGB}{0,150,0}
\definecolor{DeepPink}{RGB}{255,20,147} 
\definecolor{DarkRed}{RGB}{150,05,05}
\definecolor{VioletBlue}{RGB}{90,41,222}

\usepackage[colorlinks]{hyperref} 
\hypersetup{
   colorlinks, linkcolor={DarkGreen},
   citecolor={DarkGreen}, urlcolor={DarkRed}
}

\makeatletter
\def\@footnotecolor{red}
\define@key{Hyp}{footnotecolor}{%
 \HyColor@HyperrefColor{#1}\@footnotecolor%
}
\def\@footnotemark{%
    \leavevmode
    \ifhmode\edef\@x@sf{\the\spacefactor}\nobreak\fi
    \stepcounter{Hfootnote}%
    \global\let\Hy@saved@currentHref\@currentHref
    \hyper@makecurrent{Hfootnote}%
    \global\let\Hy@footnote@currentHref\@currentHref
    \global\let\@currentHref\Hy@saved@currentHref
    \hyper@linkstart{footnote}{\Hy@footnote@currentHref}%
    \@makefnmark
    \hyper@linkend
    \ifhmode\spacefactor\@x@sf\fi
    \relax
  }%
\makeatother
\hypersetup{footnotecolor=DarkGreen}

\usepackage{amsthm} 
\theoremstyle{plain} 
\newtheorem*{corr*}{Corollary} 
\newtheorem{tthm}{Theorem}
\newtheorem{tthmm}{Theorem}
\theoremstyle{definition} 
\newtheorem{thm}{Theorem}[section]
\newtheorem{prop}[thm]{Proposition} 
\newtheorem*{ack}{Acknowledgments}
\newtheorem*{out}{Outline of paper}
\newtheorem*{notn}{Notation}
\newtheorem*{conv}{Convention Choice}
\newtheorem{dfn}[thm]{Definition} 
\newtheorem{cor}[thm]{Corollary}  
\newtheorem{ex}[thm]{Example}
\newtheorem{rem}[thm]{Remark}
\newtheorem{lem}[thm]{Lemma}

\newcommand{\an}{ \text{an} } 

\newcommand{\ev}{\text{ev}}
\newcommand{\ps}{k\{\!\{u\}\!\}}
\newcommand{\gnor}{|\!|\!\cdot \!|\!|}
\newcommand{\abs}[1]{\left| \! \left| {#1} \right| \! \right|}

\def\semicolon{;}
\def\applytolist#1{
    \expandafter\def\csname multi#1\endcsname##1{
        \def\multiack{##1}\ifx\multiack\semicolon
            \def\next{\relax}
        \else
            \csname #1\endcsname{##1}
            \def\next{\csname multi#1\endcsname}
        \fi
        \next}
    \csname multi#1\endcsname}

\usepackage{amssymb} 
\def\calc#1{\expandafter\def\csname b#1\endcsname{{\mathbb #1}}}
\applytolist{calc}QWERTYUIOPLKJHGFDSAZXCVBNMqwertyuiopasdfghjklzxcvbnm;
\def\calc#1{\expandafter\def\csname bf#1\endcsname{{\mathbf #1}}}
\applytolist{calc}QWERTYUIOPLKJHGFDSAZXCVBNMqwertyuiopasfghjklzxcvbnm;
\def\calc#1{\expandafter\def\csname c#1\endcsname{{\mathcal #1}}}
\applytolist{calc}QWERTYUIOPLKJHGFDSAZXCVBNM;
\def\calc#1{\expandafter\def\csname s#1\endcsname{{\mathscr #1}}}
\applytolist{calc}QWERTYUIOPLKJHGFDSAZXCVBNM;
\def\calc#1{\expandafter\def\csname f#1\endcsname{{\mathfrak #1}}}
\applytolist{calc}QWERTYUIOPLKJHGFDSAZXCVBNMpqrstmnxyuv;
\def\calc#1{\expandafter\def\csname tb#1\endcsname{{\text{\textbf{#1}}}}}
\applytolist{calc}QWERTYUIOPLKJHGFDSAZXCVBNMqwertyuiopasdfghjklzxcvbnm;

\DeclareMathOperator{\spec}{Spec}
\DeclareMathOperator{\spn}{span}
\DeclareMathOperator{\trop}{Trop}

\DeclareMathOperator{\trp}{trop}

\title{Spherical tropicalization and Berkovich analytification}
\author{Desmond Coles}
\date{ \small \textbf{The University of Texas at Austin \\
Austin, TX, United States of America \\
dcoles@utexas.edu
}
}

\begin{document}
\maketitle
\noindent
   \textbf{Abstract.} Let $X$ be a spherical variety. We show that Tevelev and Vogiannou's tropicalization map from $X$ to its tropicalization factors through the Berkovich analytification $X^{\an}$, as in the case for toric varieties. Furthermore we show that the tropicalization is a strong deformation retraction of $X^{\an}$. We also give a strong deformation retraction of Thuillier's analytification $X^{\beth}$ onto a subspace described using the colored fan of $X$.
   \vspace{\baselineskip}

\vspace{\baselineskip}

\noindent
{\textbf{Statements and Declarations.}} The author has no competing interests.
\vspace{\baselineskip}

\noindent
{\textbf{Acknowledgements.}} The author would like to thank their advisor, Sam Payne, for their support and guidance during this research. The paper also benefited greatly from conversations with Tom Gannon, Sam Raskin, Yixian Wu, Hernán Iriarte, Amy Li, and Logan White. The author was grateful to be funded by NSF DMS–2001502 and NSF DMS–2053261.


\pagebreak

\section{Introduction}

Tropical geometry provides a set of tools for assigning combinatorial objects to algebraic varieties. One then studies the algebraic geometry of the original variety by looking at the combinatorics. One well known example is given by embedding a variety into a toric variety and then tropicalizing the toric variety \cite{Analytification}. In \cite{TV} Tevelev and Vogiannou introduce tropicalization for spherical homogenous spaces. Let $k$ be an algebraically closed field. Recall that if $G$ is a reductive group over $k$, and $X$ is a normal $G$-variety, then $X$ is \textbf{spherical} if there is a Borel subgroup of $G$ with an open orbit in $X$. Spherical varieties can be considered the nonabelian analogue of toric varities. Let  $\overline{K}=\ps$ be Puiseux series, the algebraic closure of Laurent series $K=k(\!(u)\!)$. Equip $\overline{K}$ with the $u$-adic valuation $\gnor_u$.  Given a spherical homogeneous space $G/H$, Tevelev and Vogiannou define a set theoretic tropicalization map $\trp_G\colon G/H(\overline{K})\rightarrow \cV$, where $\cV$ is the set of $G(k)$-invariant valuations on the function field $k(G/H)$. The set $\cV$ in fact has the structure of a finitely generated convex cone \cite[Corollary 5.3]{Knop}. The map $\trp_G$ takes $x\in G/H(\overline{K})$ to the valuation $\gnor_{x}\in\cV$ defined by the formula:
\[
    \abs{f}_x=\abs{gf(x)}_u \quad \text{$g$ in general position}
\]
where general position means $g\in U_f$, $U_f\subseteq G(k)$ a nonempty Zariski open such that $\abs{gf(x)}_u$ is constant for $g\in U_f$. For a subvariety $W\subseteq G/H$ one can then define $\trop_G(W)$ as the closure of $\trp_G(W(\overline{K}))$. This construction agrees with the usual tropicalization map in the case when $G=T$, $T$ a torus, and $H$ trivial. 

In \cite{Nash1} Tevelev and Vogiannou's construction was extended to all spherical varieties, generalizing the extension of tropicalization of tori to toric varities given in \cite{Analytification}. Here the target of the tropicalization map, $\trop_G(X)$, consists of a gluing of the valuation cones associated to $G$-orbits in $X$. The tropicalization map is defined as a disjoint union of the maps given above. In \cite{KM} Kaveh and Manon develop tropical geometry for spherical varieties via Gr\"{o}bner theory and it is shown that this agrees with the above notion.

Let $X^{\an}$ be the Berkovich analytification of $X$. In the case when $X$ is a toric variety with dense torus $T$, $\trop_T(X)$ is canonically a subspace of $X^{\an}$ and the tropicalization map factors through a continuous retraction $X^{\an}\rightarrow \trop_T(X)$. It was conjectured in \cite{TV} that this is true for tropicalizations of spherical homogeneous spaces. In \cite{KM} this is proved for affine spherical varieties. We show this is the case for all spherical varieties. Let $G^{\beth}$ denote the $k$-analytic space associated to $G$ as defined in \cite[Section 1]{Thuillier}. Let $\bfp\colon X^{\an}\rightarrow X^{\an}$ be the map given by $\ast$-multiplication by the unique point in the Shilov boundary of $G^{\beth}$.

\begin{tthm}\label{IntroAthm}
Let $X$ be a spherical $G$-variety over an algebraically closed and trivially valued field $k$, and let $\overline{K}$ denote Puiseux series in $k$. Then there is a canonical homeomorphism $\iota:\trop_G(X)\rightarrow \bfp(X^{\an})$, and $\bfp\colon X^{\an}\rightarrow \trop_G(X)$ is a retraction of topological spaces such that the following diagram commutes:
\[
\begin{tikzcd}
X(\overline{K}) \arrow[r] \arrow[dr, "\trp_G"'] & X^{\text{an}} \arrow[d, "\bfp"]  \\
 \, &  \trop_G(X).
 \end{tikzcd}
\]
\end{tthm} 

\begin{corr*}
For $W$ a subvariety of $X$, there is a strong deformation retraction $H:[0,1]\times W^{\an}\rightarrow \trop_G(W)$ such that $H(1,p)=\bfp(p)$.
\end{corr*}

The construction of the map $\bfp$ was inspired by the work of Thuillier in the toric case, \cite[Section 2]{Thuillier}. The space $G^{\beth}$ is a compact analytic subdomain of the $k$-analytic space $G^{\an}$, the Shilov boundary of this space is a point with important similarities to the generic point of $G$. The notion of $\ast$-multiplication is an idea originally due to Berkovich which generalizes left-multiplication by points in $G^{\an}(k)$. 

Our work in Section \ref{starmult sec} extends the notion of tropicalization to other nontrivially valued extensions of $k$ besides $\overline{K}$. We show that for any field extension $L/k$ with valuation $\gnor$ there is a well-defined tropicalization map $\trp_{G,L}\colon X(L)\rightarrow \trop_G(X)$. For $x \in X(L)$, $\trp_G(x)= \gnor_{x}$ where $\gnor_{x}$ is the valuation $\gnor_{x}\in\cV$ given by the formula:
\begin{equation*}
\abs{f}_x=\abs{gf(x)} \quad \text{$g$ in general position}
\end{equation*} where general position means $g\in U_f$, $U_f\subseteq G(k)$ is a nonempty Zariski open such that $\abs{gf(x)}$ is constant for $g\in U_f$. It follows from Theorem \ref{IntroAthm} that tropicalization is invariant under the choice of $L$.

\begin{corr*}\label{introthmAcor cor}
Let $W$ be a subvariety of $X$, and $L$ be a nontrivially valued and algebraically closed extension of $k$. Then the closure of $\trp_{G,L}(W(L))$ in $\trop_G(X)$ is equal to $\bfp(W^{\an})$, where $W^{\an}$ is considered as a subspace of $X^{\an}$ induced by the inclusion $W\subseteq X$.
\end{corr*}

Our second main result is on the restriction of the map $\bfp$ to the subspace $X^{\beth}$. This result generalizes Proposition 2.12 of \cite{Thuillier} on the toric case. 

\begin{tthm}\label{IntroBthm}
There is a strong deformation retraction from $X^{\beth}$ onto $\bfp(X^{\beth})$. If the colored fan $\fF(X)$ satisfies:
\begin{equation*}\label{into thm b eqn}
\bigcup_{(\sigma,\cF)\in \fF(X)}\sigma \subseteq \cV.  \tag{$\star$}
\end{equation*}
Then $\bfp(X^{\beth})$ is homeomorphic to the the canonical compactification of the fan $F(X)=\{\sigma \, | \, (\sigma,\cF)\in \fF(X)\}$.
\end{tthm}

Recall that just as toric varieties are classified by fans, spherical varieties are classified by colored fans. For each spherical homogeneous space $G/H$ there is an associated vector space $N_{\bR}$. A colored fan, defined precisely in Section \ref{Colored fans sec}, is a finite collection of pairs $(\sigma, \cF)$ where $\sigma$ is a strictly convex cone in $N_{\bR}$ and $\cF$ is a finite set. The \textbf{canonical compactification} of a cone $\sigma$ is defined as $\hom(\sigma^{\vee},[0,\infty])$, where $\sigma^{\vee}$ consists of positive linear functionals on $\sigma$ and the homomorphisms are semigroup homomorphisms that preserve scaling by $\bR_{\geq 0}$. The canonical compactification of a fan is the polyhedral complex given by gluing the canonical compactifications of the cones along shared faces. See Remark \ref{star cond rem} for an explanation of why we impose \ref{into thm b eqn} and some examples where this condition does and does not hold. In Remark \ref{description when star not satisfied rem} we discuss how to give a description of $\bfp(X^{\beth})$  when \ref{into thm b eqn} does not hold.

\begin{out}
Section \ref{Berkovichgeo Sec} introduces the necessary notions from Berkovich geometry. Section \ref{starmult sec} introduces $\ast$-multiplication and studies the map $\bfp$. Section \ref{Colored fans sec} recalls the Luna-Vust classification of $G/H$-embeddings, and contains important notation for Sections \ref{tropicalization section} and \ref{Generalizing Thuillier sec}. Section \ref{tropicalization section} recalls the work of \cite{TV} and \cite{Nash1} and gives a proof of Theorem \ref{IntroAthm}. Section \ref{Generalizing Thuillier sec} recalls the work of \cite[Section 2]{Thuillier} and gives a proof of Theorem \ref{IntroBthm}.
\end{out}

\begin{conv}
Though somewhat nonstandard in the literature of spherical varieties we opt to use multiplicative notation for valuations to accommodate the conventions of Berkovich geometry. Our exact conventions are as follows. Let $A$ be a ring. By a \textbf{seminorm} on $A$ we mean a function $\gnor_A\colon A\rightarrow \bR_{\geq 0}$ such that $\abs{1}_A=1$, $\abs{0}_A=0$, $\abs{fg}_A\leq \abs{f}_A\abs{g}_A$, and $\abs{f+g}_A\leq \max({\abs{f}_A,\abs{g}_A})$.\footnote{These are non-Archimedean seminorms, which are the only ones of interest because our base field will always be non-Archimedean.} We say that $\gnor_A$ is a \textbf{norm} if $\abs{f}_A\neq 0$ for all $f\neq 0$. A (semi)norm will be called a \textbf{(semi)valuation} if it is multiplicative: $\abs{fg}_A=\abs{f}_A\abs{g}_A$.\footnote{More commonly valuations given as functions $\nu$ into the additive semigroup $\bR\sqcup\{\infty\}$, a multiplicative norm is equivalent to such a function $\nu$, set $\nu=-\log(\gnor)$.} If $k$ is a field with valuation $\gnor_k$ and $A$ is a $k$-algebra then a seminorm on $A$ will always be assumed to extend the valuation on $k$; $\abs{\lambda f}_A=\abs{\lambda}_k\abs{f}_A$ for all $\lambda\in k$. By a valued extension $L/k$ we mean a field extension where $L$ is equipped with a valuation (extending the valuation on $k$ by our previous convention). 
\end{conv}

\begin{ack}
The author would like to thank their advisor, Sam Payne, for their support and guidance during this research. The paper also benefited greatly from conversations with Tom Gannon, Sam Raskin, Yixian Wu, Hernán Iriarte, Amy Li, and Logan White.
\end{ack}

\section{Overview of Berkovich geometry}\label{Berkovichgeo Sec}

The contents of this Section are an overview of the necessary notions in the theory of Berkovich spaces, in particular we need to work with analytifications of varieties and $\beth$-spaces. Berkovich spaces provide a framework for doing analytic geometry over valued fields. Given a base field $k$ with valuation, Berkovich spaces are called $k$-analytic spaces. These are locally ringed spaces equipped with a $k$\textit{-analytic atlas}. We will not go into full detail of the theory as we are primarily concerned with the topology of these spaces, and will only introduce the notions we need. The content in this Section is laid out in detail in \cite{BerkBook}, with the exception of $\beth$-spaces, for this we refer the reader to \cite{Thuillier}. 

Throughout this Section let $k$ be a valued field, and let $X$ be a separated locally finite type $k$-scheme.

\subsection{Berkovich analytification}

\begin{dfn}
The \textbf{Berkovich analytification} of $X$, denoted $X^{\an}$, has points given by pairs $(x,\gnor)$ where $x$ is a scheme point of $X$ and $\gnor$ is a valuation on $k(x)$. There is a map $\rho\colon X^{\an}\rightarrow X$ given by $(x,\gnor)\mapsto x$. The topology on $X^{\an}$ is the coarsest such that $\rho$ is continuous, and for any open $U\subseteq X$ and any $f\in \cO_X(U)$, the map $\rho^{-1}(U)\rightarrow \bR$ given by $(x,\gnor)\mapsto \abs{f}$ is continuous. 
\end{dfn}

Any valued extension $L/k$ and map of schemes $x\colon \spec L \rightarrow  X$ defines a point in $X^{\an}$. The point is given by $(x,\gnor)$ where $x$ is the point corresponding to the image of $\spec L$ and $\gnor$ is the valuation on $L$ restricted to $k(x)$. There is also a sheaf of analytic functions on $X^{\an}$, however we only need a few details. In particular for $p\in X^{\an}$ the stalk at $p$ is a local ring. If $\kappa(p)$ is the residue field of the stalk at $p$ then $\kappa(p)$ has a norm and we let $\cH(p)$ denote the completion with respect to this norm. If $p=(x,\gnor)$ then there is a canonical isometric isomorphism from the completion of $k(x)$ with respect to $\gnor$, to $\cH(p)$. 

When $X=\spec A$, $A$ a $k$-algebra, we can alternatively describe $X^{\an}$ as being the semivaluations on $A$. The topology on $X^{\an}$ in this case will be the coarsest such that evaluation at an element of $A$ is continuous. One can see that this description of $X^{\an}$ is equivalent to the original as follows. Let the \textbf{kernel} of a semivaluation $\gnor$ on $A$, be $\ker (\gnor):=\{f\in A \, | \, \abs{f}=0\}$. Then $\gnor$ defines a valuation on the fraction field of $A/\ker (\gnor)$, which we also denote $\gnor$. This valuation is characterized by $\abs{f+\ker(\gnor)}:=\abs{f}$.

Berkovich analytification is functorial. For any morphism of $k$-schemes $f\colon X\rightarrow Y$ there is a morphism of Berkovich spaces over $k$, $f^{\an}\colon X^{\an}\rightarrow Y^{\an}$. In particular for a point $(x,\gnor)\in X^{\an}$ one has that $f^{\an}((x,\gnor))=(f(x),\abs{f_x(\cdot)})$, where $f_x \colon    k(f(x))\rightarrow k(x)$ is the induced map on function fields.

\subsection{The Berkovich spectrum and $\beth$-spaces}\label{beth spaces subsec}

 We need to introduce one other type of $k$-analytic space. Let $\cA$ be a Banach $k$-algebra, i.e. $\cA$ is a $k$-algebra equipped with a norm $\gnor_{\cA}$, and $\cA$ is complete with respect to this norm.
 
\begin{dfn}
The \textbf{Berkovich spectrum}, $\cM(\cA)$, is the set of bounded semivaluations on $\cA$, i.e. semivaluations $\gnor$ such that where $\gnor \leq \gnor_{\cA}$. The topology is the coarsest such that evaluation at an element of $\cA$ is continuous.
\end{dfn}

\begin{ex}
The one-point spaces in the category of $k$-analytic spaces are given by $\cM(L)$ where $L$ is a complete field over $k$. 
\end{ex}

Let $\cA$ and $\cB$ be Banach $k$-algebras with norms $\gnor_{\cA}$ and $\gnor_{\cB}$. One defines the completed tensor product, $\cA\hat{\otimes}_k\cB$ to be the completion of $\cA\otimes_k \cB$, with respect to the norm:
\begin{equation*}\label{prodnorm Eqn}
|\!|f|\!|_{\cA, \cB}: =\inf\left\{\max_i |\!|g_i |\!|_{\cA} |\!|f_i |\!|_{\cB}\, \bigg| \, \text{$\sum\limits_i f_i\otimes g_i\in \cA \otimes_k \cB$ such that} \,  f=\sum\limits_i f_i \otimes g_i\right\}.
\end{equation*}
The function $\gnor_{\cA{\otimes}_k \cB}$ is clearly a seminorm, that it is a norm and that $\cA{\otimes}_k \cB$ embeds isometrically into its completion follows from part 4 of Theorem 1 in \cite[Section 3]{Gr}.

\begin{prop}\cite[Chapter 3]{BerkBook}\label{norm makes product prop}
The norm $\gnor_{\cA,\cB}$ makes $\cM(\cA\hat{\otimes}_k\cB)$ the product $\cM(\cA)\times \cM(\cB)$ in the category of $k$-analytic spaces. 
\end{prop}

\begin{notn}
To avoid cluttered subscripts, when considering $p$ in some analytic space will choose to write $\gnor_p$ for the norm on $\cH(p)$ instead of $\gnor_{\cH(p)}$. Similarly for a pair of points $p$ and $q$ we will write $\gnor_{p,q}$ for the norm on $\cH(p)\hat{\otimes}_{k}\cH(q)$ instead of $\gnor_{\cH(p),\cH(q)}$. 
\end{notn}

Consider when $k$ is trivially valued. Let $A$ be a finitely generated $k$-algebra, $A$ equipped with the trivial valuation is complete. In this case we can identify $\cM(A)$ with seminorms $\gnor$, that are bounded by $1$. Recall that if $X$ is a variety over $k$, $L/k$ is a valued field extension, and $R$ is the valuation ring of $L$, then the map of schemes $\spec L \rightarrow X$ \textbf{has center} if it factors through $\spec L\rightarrow \spec R$.  We say \textbf{the center} is the image of the closed point of $\spec R$. Furthermore if $X=\spec A$ and $A$ is an integral domain, then a valuation $\gnor$ has center in $X$ if and only if $\gnor$ is bounded by 1 on $A$. From this it follows that for any trivially valued finitely generated $k$-algebra $A$, $\cM(A)$ can be taken to be pairs $(x,\gnor)$, $x$ a scheme point of $A$, and $\gnor$ a valuation on $k(x)$ having center in $\spec A$. For a general locally finite type scheme $X$ over a trivially valued field $k$, we define $X^{\beth}$ to be given by pairs of points $(x,\gnor)$ where $\gnor$ is a valuation on $k(x)$, and $\spec k(x) \rightarrow X$ has center in $X$. So $\cM(A)=\left(\spec A\right)^{\beth}$. 

\begin{prop}\cite{Thuillier}\label{charcterization of xbeth prop}
The space $X^{\beth}$ is a compact analytic subdomain of $X^{\an}$, and they are equal when $X$ is proper.
\end{prop}

\subsection{The reduction map}
Let $\cA^{\circ}:=\{f\in \cA\, | \, \abs{f}_{\cA}\leq 1\}$ and $\cA^{\circ \circ}:=\{f\in \cA\, | \, \abs{f}_{\cA}< 1\}$. Then the \textbf{reduction} of $\cA$ is said to be $\widetilde{\cA}: = \cA^{\circ}/\cA^{\circ \circ}$. Any map of Banach $k$-algebras $\cA\rightarrow \cB$ induces a map of rings $\widetilde{\cA}\rightarrow \widetilde{\cB}$. In particular for $p\in \cM(\cA)$ there is a map of rings $h_p\colon \widetilde{\cA}\rightarrow \widetilde{\cH(p)}$. Then there is a map of sets called the \textbf{reduction map}, $r\colon \cM(\cA)\rightarrow \spec(\widetilde{\cA})$ defined by $p\mapsto \ker(h_p)$.

When $A$ is a finitely generated algebra with the trivial norm we have $A^{\circ}=A$ and $A^{\circ \circ}=0$, so reduction gives a map of sets $r\colon \cM(A)\rightarrow \spec A$. More generally we have a map $r\colon X^{\beth}\rightarrow X$. When $X^{\beth}$ is viewed as pairs $(x,\gnor)$, reduction maps $(x,\gnor)$ to its center. This map is anticontinuous \cite[Corollary 2.4.2]{BerkBook}, and surjective \cite[Proposition 2.4.4]{BerkBook}. Recall that anticontinuous means that the preimage of an open set is closed.

\subsection{The Shilov boundary}

A \textbf{boundary} of $\cM(\cA)$ is a closed subset of $\cM(\cA)$, such that for every $f\in \cA$ the function $\ev_{f}\colon \cM(\cA)\rightarrow \bR$ given by $\gnor \mapsto |\!|f|\!|$, attains its maximum. A unique minimal boundary is said to be the \textbf{Shilov boundary}. \begin{prop}\cite[Proposition 2.4.4]{BerkBook}\label{Shilovboundary and reduction prop}
The reduction map $r\colon \cM(\cA)\rightarrow \spec (\widetilde{\cA})$ gives a bijection between the Shilov boundary of $\cM(\cA)$ and the generic points of the components of $\spec (\widetilde{\cA})$.
\end{prop}

\section{The retraction map}\label{starmult sec}

In this Section we assume $k$ is a trivially valued and algebraically closed field, $G$ is a connected linear algebraic group over $k$, and $X$ is a $G$-variety. Let $m$ to denote the action map $m\colon G\times X \rightarrow G$. The objective is to introduce $\ast$-multiplication, use it construct a retraction of topological spaces $\bfp\colon X^{\an}\rightarrow X^{\an}$, and give a description of the image of $\bfp$. Furthermore we show that when $G$ is reductive we can construct a strong deformation from $X^{\an}$ onto $\bfp(X^{\an})$ using $\ast$-multiplication.

The notion of $\ast$-multiplication is taken directly from \cite[Chapter 5]{BerkBook}. Properties specific to $\ast$-multiplication by the unique point in the Shilov boundary of $G^{\beth}$ are directly inspired by the results on tori from \cite[Section 2]{Thuillier}.

\subsection{\texorpdfstring{$\ast$}{*}-multiplication}

\begin{dfn}
We say that a commutative Banach $k$-algebra $\cA$ is \textbf{peaked} if for any valued extension $L/k$ the norm on $L\hat{\otimes}_k\cA$ is a valuation. We say a point $p$ of a $k$-analytic space is peaked if the completed residue field $\cH(p)$ is peaked.
\end{dfn}

\begin{prop}\label{PeakedShilov Prop}
All the the points of $G^{\an}$ are peaked.
\end{prop}
\begin{proof}
This is a particular case of \cite[Corollaire 3.14]{Angelic}.
\end{proof}

\begin{dfn}
Let $g\in G^{\an}$, because $g$ is peaked the norm $\gnor_{g,p}$ is a valuation and is thus an element of $\cM(\cH(g))\times \cM(\cH(p))$. There is a map of sets $X^{\an}\rightarrow X^{\an}$ given by $p \mapsto g\ast p$; where  $g \ast p$ is the image of $\gnor_{g,p}$ under the composition of maps of $k$-analytic spaces:
\[
\cM(\cH(g))\times \cM(\cH(p))\hookrightarrow G^{\an}\times X^{\an} \xrightarrow{m^{\an}} X^{\an}.
\] This is called \textbf{$\ast$-multiplication by $g$}.
\end{dfn}

\begin{ex}
Let $g\in G^{\an}(k)=G(k)$. The map $X^{\an}\rightarrow X^{\an}$ given by $p\mapsto g\ast p$ is equal to the map induced by left multiplication by $g$.
\end{ex}

Below are some general facts about $\ast$-multiplication.

\begin{prop}\cite[Proposition 5.2.8]{BerkBook}\label{cd prop} Let $G$ and $H$ be algebraic groups acting on varieties $X$ and $Y$, respectively. Let $\phi\colon G\rightarrow H$ be a morphism of algebraic groups, and $\psi\colon  X\rightarrow Y$ a morphism of varieties.
\begin{enumerate}
    \item  The map $X^{\an} \rightarrow X^{\an}$ given by $\ast$-multiplication by a point $g\in G^{\an}$ is continuous.
    \item If $g$ and $h$ are points of $G^{\an}$ and $H^{\an}$, respectively, and $p$ a point of $X^{\an}$, then $g\ast(h\ast p)=(g\ast h)\ast p$
    \item Given a commutative diagram:
    \[ 
\begin{tikzcd}
G^{\an} \times X^{\an} \arrow{d}{\phi^{\an} \times \psi^{\an}} \arrow{r}{} & X^{\an} \arrow{d}{\psi^{\an}} \\
H^{\an} \times Y^{\an} \arrow{r}{} & Y{\an}
\end{tikzcd}
    \]  
   one has that $\psi^{\an}(g\ast p)=\phi^{\an}(g)\ast\psi^{\an}(p)$.
\end{enumerate}
\end{prop}

Let $\bfg\in G^{\an}$ be the point given by the trivial valuation. This is the unique point in the Shilov boundary of $G^{\beth}$. Let $\bfp\colon X^{\an}\rightarrow X^{\an}$ be given by $p\mapsto \bfg\ast p$. By Proposition \ref{cd prop} part 1, $\bfp$ is a continuous map of topological spaces. The remainder of this section will be devoted to studying this map.

\begin{rem}
We can alternatively describe $\bfp$ in the following way. We have that $\bfp(p)$ is the image of the unique point in the Shilov boundary of $G^{\beth}\times \cM(\cH(p))$ under the map of $k$-analytic spaces:
\[
G^{\beth}\times \cM(\cH(p)) \hookrightarrow G^{\beth}\times X^{\an}\rightarrow X^{\an}.
\] So when $G=T$ is a torus, $\bfp$ is the retraction map defined in \cite[Section 2]{Thuillier}.
\end{rem}

\begin{ex}
Let $X=T=\spec k[t_1^{\pm 1},\ldots t_n^{\pm 1}]$ be the $n$-dimensional torus, and equip $k[t_1^{\pm 1},\ldots t_n^{\pm 1}]$ with the trivial valuation. Let $\lambda\in [0,1]$, and let $\bfg_{\lambda}$ be the point of $T^{\beth}$ corresponding to the semivaluation:
\[
|\!|\sum\limits_{I}a_It^I|\!|=\max_{I}\lambda^{|I|}
\] where $I$ denotes a multi-index and if $I=(I_1,\ldots I_n)$ then $|I|=I_1+\cdots I_n$. Then one can see via computation that for any valued extension $L$, the norm $\gnor_{\bfg_{\lambda},L}$ on $\cH(\bfg_{\lambda})\hat{\otimes}_kL$ is multiplicative. For $\lambda=0$ this is because $\cH(\bfg_\lambda)=k$. For $\lambda\in(0,1]$ $\bfg_\lambda$ defines a norm on $k[t_1^{\pm 1},\ldots t_n^{\pm 1}]$. For $f=\sum_Ia_It^I\in k[t_1^{\pm 1},\ldots t_n^{\pm 1}] \otimes_k L =L[t_1^{\pm 1},\ldots t_n^{\pm 1}]$ we compute that:
\[
|\!|f|\!|_{\bfg_{\lambda},L}=\max_{I}|\!|a_I|\!|_{L}\lambda^{|I|}
\] and this is a valuation.

Let $p\in T^{\an}$ be given by $x=(x_1,\ldots x_n)\in T(L)$, where $L$ is a valued extension of $k$, so $p=(x,\gnor_L)$. Let $m_x$ the composition $T\times \spec k(x) \hookrightarrow T\times T \xrightarrow{m} T$. Then for $f=\sum_Ia_It^I\in \spec k[t_1^{\pm 1},\ldots t_n^{\pm 1}]$ seminorm $\bfg_{\lambda}\ast p$ is given by:
\[
|\!|f|\!|_{\bfg_{\lambda}\ast p}=|\!|m_x^{\ast}f|\!|_{\bfg_{\lambda}, p}=|\!|\sum\limits_I a_I t^I \otimes x^I|\!|_{\bfg_{\lambda}, p}=\max_I|\!|a_Ix^I|\!|_L\lambda^{|I|}.
\]
\end{ex}

\begin{cor}\label{Idpotentcy Cor}
The map: $\bfp\colon  X^{\an}\rightarrow X^{\an}$ is idempotent, i.e. $\bfg\ast \bfg=\bfg$.
\end{cor}
\begin{proof}
One has that $\cH(\bfg)=k(G)$ and the norm on $k(G)$ defined by $\bfg$ is the trivial valuation so we have that for $f\in \cH(\bfg)\otimes_k\cH(\bfg)$:
\[
|\!|f |\!|_{\bfg,\bfg}=\inf\{\max_i |\!|g_i |\!|_{\bfg}|\!|f_i |\!|_{\bfg} \, | \, f=\sum\limits_i g_i\otimes f_i\}
\]
which is 0 or 1, and 0 if and only if $f=0$. Therefore it is the trivial norm. so $\bfg \ast \bfg =\bfg$. So by Proposition \ref{cd prop} $\bfp$ is idempotent.
\end{proof}

\begin{cor}\label{Retraction cor}
The map $\bfp\colon X^{\an}\rightarrow X^{\an}$ is a retraction of topological spaces.
\end{cor}

\begin{prop}\label{imageofgstar Prop}
Let $p=(x,\gnor)\in X^{\an}$, then $\bfp(p)=(\eta_{Y},\overline{\gnor})$. Where $\eta_{Y}$ is the generic point of $Y$, $Y$ is the Zariski closure of the image of the composition:
\[
G\times \spec(k(x)) \hookrightarrow G\times X \xrightarrow{m} X
\]
and $\overline{\gnor}$ is a $G(k)$-invariant valuation on $k(Y)$.
\end{prop}

The proof is similar to the proofs of \cite[Proposition 2.5 (2)]{Thuillier} and \cite[Proposition 2.3 (ii)]{Thuillier}, which prove the above for $G=T$ a torus.

\begin{proof}
Recall that there is a map $\rho:X^{\an}\rightarrow X$ given by $(x,\gnor)\mapsto x$. To see that $\rho(\bfp(p))=\eta_Y$ notice that there is a commutative diagram:
\[
\begin{tikzcd}
\cM(\cH(\bfg))\times\cM(\cH(p)) \arrow[d,hookrightarrow] \arrow{r}{\rho} & \spec k(\eta_{G})\times \spec k(x) \arrow[d,hookrightarrow]\\
G^{\an}\times X^{\an} \arrow{d}{m^{\an}} \arrow{r}{\rho} & G\times X \arrow{d}{m}\\
X^{\an} \arrow{r}{\rho} & X
\end{tikzcd}
\]
where $\eta_G$ is the generic point of $G$. In this diagram the image of the generic point of $\spec k(\eta_{G})\times \spec k(x)$ in $X$ is $\eta_Y$. But then because $\rho(\gnor_{\bfg,p})$ is the generic point of $\spec k(\eta_{G})\times \spec k(x)$ we have the first claim.

To see that $\overline{\gnor}$ is $G(k)$-invariant notice that for any $g\in G(k)$, $g \ast \bfg=\bfg$. Thus by Proposition \ref{cd prop} part 2 we have that $g\ast (\bfg \ast p)=\bfg\ast p$. But $p\mapsto g\ast p$ is the map induced by left multiplication by $g$, which for points of the form $(\eta_Y,\gnor)$ is given by precomposing $\gnor$ with the left-multiplication by $g$.
\end{proof}

Furthermore we have the following.

\begin{prop}\label{infinite Prop} The image of $\bfp$ is exactly the set of points $(\eta_Y,\gnor)$, where $Y$ is a $G$-invariant subvariety of $X$ and $\gnor$ is $G(k)$-invariant.
\end{prop}

This proof is the same as the proof of \cite[Proposition 2.3 (ii)]{Thuillier}, which proves this fact for $G=T$ a torus. We include the argument, except for one lemma, for completeness.

\begin{proof} By Proposition \ref{imageofgstar Prop} we have one containment. For the other notice that points of the form $(\eta_Y,\gnor)$, where $Y$ is a $G$-invariant subvariety and $\gnor$ is $G(k)$-invariant, are exactly the points $(x,\gnor)\in X^{\an}$ fixed by the action of $G(k)$. Let $p=(x,\gnor)\in X^{\an}$ be a point fixed by the group action $G(k)$. The morphism:
\[
G^{\beth}\times \cM(\cH(p))\hookrightarrow G^{\beth}\times X^{\text{an}}\xrightarrow{m^{\an}} X^{\text{an}}
\]
maps the subset $G^{\beth}(k)\times \cM(\cH(p))$ onto $x$. The map is continuous and $x$ is closed as Berkovich spaces are locally hausdorff and thus T1, and so every point in the closure of $G^{\beth}(k)\times \cM(\cH(p))$ will have image $p$. Because $\bfp(p)$ is the image in $X^{\an}$ of the unique Shilov boundary point of $G^{\beth}\times \cM(\cH(p))$ under \[
G^{\beth}\times \cM(\cH(p))\hookrightarrow G^{\beth}\times X^{\text{an}}\xrightarrow{m^{\an}} X^{\text{an}}
\] we need to show that the unique point in the Shilov boundary of $G^{\beth}\times \cM(\cH(p))$ is in the closure of $G^{\beth}(k)\times \cM(\cH(p))$. The field $k$ is algebraically closed so $G(k)\times \spec \widetilde{\cH(p)}$ is a dense subset of $G\times \spec \widetilde{\cH(p)}$. The following lemma will complete the proof.
\end{proof}

\begin{lem}
Let $L/k$ be a nonarchimedean field extension, and let $\Sigma$ be a subset of $G^{\beth}\times_k \cM(L)$ with $r(\Sigma)$ dense in $G\times_k\spec(\tilde{L})$. Then the Shilov boundary is in the closure of $\Sigma$.
\end{lem}
\begin{proof}
The case when $G$ is a torus is proved in \cite[lemma 2.4]{Thuillier}, and the proof proceeds the same for any connected linear algebraic group $G$.
\end{proof}

\subsection{The action of \texorpdfstring{$\bfp$}{p} on homogeneous spaces}

If $G$ is reductive and $X$ is a spherical $G$-variety, then any $G$-invariant subvariety is in fact the closure of the $G$-orbit of some closed point $y\in X$ (\cite{Knop}). Thus we will restrict our attention to when $G$ acts on an orbit. In this case we will give a formula for $\bfp$ in terms of valuations that will immediately relate it to the tropicalization map. We begin by recalling the construction of Luna and Vust which is used in \cite{TV} to define tropicaliziation.

Let $\overline{K}=\ps$ denote Puiseux series over $k$, equipped with the $u$-adic valuation $\gnor_u$. In the case that $G/H$ is spherical for any $\overline{K}$-point $x$, Luna and Vust define a $G(k)$-invariant valuation $\gnor_x$ (\cite[Section 4]{LV}). For any $f\in k(G/H)$ there is a nonempty open subset of $U\subseteq G(k)$ such that if $g\in U$ then $gf(x)$ is well defined. Luna and Vust show that for any $f\in k(G/H)$ there is a dense open $U_f\subseteq G(k)$ where $\abs{gf(x)}_u$ is constant as $g$ varies over $U_f$. The valuation $\gnor_x$ is given by $\abs{f}_x=\abs{gf(x)}_u$ for $g\in U_f$.

We will now explain how this generalizes to other valued extensions $L/k$. Let $x\colon \spec L \rightarrow G$ and denote the valuation on $L$ by $\gnor$. The proof of the following lemma uses the same technique as the proof of \cite[Lemma 1.4]{Knop}.
\begin{lem}\label{OpenWhereMinIsAttained Lem}
For each $f\in k(G)$ there is a nonempty Zariski open $U_f \subseteq U$ such that for any $g\in U_f$, $\abs{gf(x)}$ is constant. For $f\in k[G]$ we have $\abs{gf(x)}=\max_{h\in G(k)}\abs{hf(x)}$ when $g\in U_f$.
\end{lem}

\begin{proof}
Without loss of generality we can assume $f\in k[G]$. Let $M$ be the $G$-module generated by $f$, this is a finite dimensional vector space and thus has basis $g_1f,\ldots g_rf $. Let $v=\max_{i}\abs{g_if(x)}$. Let $V(v)=\{f'\in k[G]\, | \, \abs{f'(x)}<v\}$. Then set $U_f=\{g\in G(k)\, | \, gf\notin V(v)\cap M\}$.
\end{proof}

Let $\gnor_x:k(G)\rightarrow (0,\infty)$ be given by:
\begin{equation}\label{TheVal Eqn}
f \mapsto \abs{gf(x)} \quad g\in U_f.
\end{equation}

\begin{cor}\label{subscript val existence cor}
The function $\gnor_x$ defines a $G(k)$-invariant valuation on $k(G)$.
\end{cor}

\begin{lem}\label{EqualityofmyvalandTheirval Lem}
Let $x:\spec L \rightarrow G$ be as above and let $p=(x,\gnor)$ be the corresponding point in $G^{\an}$. Then if $\bfp(p)=(\eta_G,\overline{\gnor})$, we have that $\overline{\gnor}=\gnor_x$.
\end{lem}

\begin{proof}
We will show that $\gnor_x$ and $\overline{\gnor}$ agree for all regular functions $f\in k[G]$. Let $m_x$ be the composition $G\times \spec k(x)\hookrightarrow G\times G \xrightarrow{m} G $. Then if $m_x^*f =\sum_i g_i\otimes f_i(x)\in k[G]\otimes k(x)$ we have $\abs{f}_x=\abs{gf(x)}=\abs{\sum_i g_i(g)\otimes f_i(x)}\leq \max_i(\abs{f_i(x))})$, where $g\in U_f$, $U_f$ the open subset of $G(k)$ from Lemma \ref{OpenWhereMinIsAttained Lem}. But then $\overline{\abs{f}}=\abs{m_x^*f}_{\bfg,p}$ and
\[
\abs{m_x^*f}_{\bfg,p}=\inf\left\{\max_i\abs{g_i}_{\bfg} \abs{f_i(x)}_{p}\,\, \bigg| \,\, \sum\limits_ig_i\otimes f_i(x)=m_x^*f\right\}.
\] Because $\gnor_{\bfg}$ is trivial we have that $\abs{f}_x\leq \overline{\abs{f}}$.

Conversely, to show that $\abs{f}_x\geq \overline{\abs{f}}$ it suffices to show that there is one expression $m_x^*f =\sum_i g_i\otimes f_i(x)$ such that $\abs{f}_x=\max_i\abs{f_i(x)} $. Let $m_x^*f =\sum_i g_i\otimes f_i(x)$, by manipulating the sum we can assume all the $g_i$ and $f_i(x)$ are nonzero and the functions $g_i$ are $k$-linearly independent. We can also assume that $\max_i(\abs{f_i(x)})=\abs{f_1(x)}$. Let $V$ be the $k$-span in $L$ of $f_i(x)$. Fix $\alpha>1$. It follows from \cite[Subsection 2.6.2, Proposition 3]{BGR}\footnote{While \cite{BGR} works with $k$ nontrivially valued, this result holds for trivially valued $k$ as well, the proof is effectively the same. Such bases are called $\alpha$-Cartesian bases.} there is a $k$-basis of $V$, $v_1,\ldots v_r$ such that $v_1$ is a multiple of $f_1(x)$, and for any $\lambda_1,\ldots \lambda_r\in k$
\[
\abs{\sum\limits_j \lambda_j v_j}\geq \alpha^{-1}\max_j(\abs{v_j}).
\]
Then we have for each $i$, $f_i(x)=\sum_j\lambda^i_jv_j$. So then $m_x^*f=\sum_j(\sum\limits \lambda^i_j g_i)\otimes v_j$, by assumption for each $j$, the function $\sum_i \lambda^i_j g_i\in k[G]$ is not identically 0. Take $g\in U_f$ and such that $ \sum_i \lambda^i_1 g_i(g)\neq 0$, which exists because $U_f$ is nonempty and open and so is the set of points such that $ \sum_i \lambda^i_1 g_i(g)\neq 0$. Now we have that:
\[
\abs{f}_x=\abs{\sum\limits_j(\sum\limits_i \lambda^i_j g_i(g))\otimes v_j}\geq \alpha^{-1}\max_j(\abs{v_j})=\alpha^{-1}\abs{f_1(x)}.
\]
But this holds for all $\alpha>1$. So $\abs{f}_x=\abs{f_1(x)}=\max_i(\abs{f_i(x))})$. So $\abs{f}_x\geq \overline{\abs{f}}$, and thus the valuations are equal on $k[G]$, and thus are equal on all of $k(G)$.
\end{proof}

Let $Y$ be the $G$-orbit in $X$ of some closed point $y$. The morphism $G\rightarrow Y$ given by $g\mapsto gy$ defines an inclusion $k(Y)\rightarrow k(G)$. Given some valued extension $L/k$ and $x\in G(L)$, one can restrict $\gnor_x$ to $k(Y)$. Let $x'$ be the image of $x$ under the map $G(L)\rightarrow Y(L)$. Then $\gnor_x$ is in fact given by
\begin{equation}\label{subscript val on orbit eqn}
\abs{f}_x=\abs{gf(x')} \quad g\in U_f.
\end{equation}
We denote the restriction of $\gnor_x$ to $k(Y)$ by $\gnor_{x'}$.

\begin{cor}\label{orbit valuations cor}
The map $\bfp\colon  Y^{\an}\rightarrow Y^{\an}$ is given by $(x,\gnor)\mapsto (\eta_Y,\gnor_{x'})$. Furthermore all the $G(k)$-invariant valuations on $k(Y)$ have the form $\gnor_{x'}$ for some $x'\in Y$.
\end{cor}

\begin{proof}
We have a surjection of sets $G^{\an}\rightarrow Y^{\an}$ and a commutative diagram of maps of analytic spaces:
\[ 
\begin{tikzcd}
G^{\an} \times G^{\an} \arrow{d}{} \arrow{r}{} & G^{\an} \arrow{d}{} \\
G^{\an} \times Y^{\an} \arrow{r}{} & Y^{\an}
\end{tikzcd}
\]
Thus by Proposition \ref{cd prop} part 3 we see that $\bfp((x,\gnor))=(\eta_Y,\gnor_{x'})$. 

To see that this is all $G(k)$-invariant valuations on $k(Y)$, notice that it follows from Proposition \ref{infinite Prop} that all $G(k)$-invariant valuations on $k(G)$ are of the form $\gnor_x$ for some $x\in X(L)$ and it follows from \cite[Corollary 1.5]{Knop} that any $G(k)$-invariant valuation on $k(Y)$ is the restriction of one on $k(G)$.
\end{proof}

\subsection{A strong deformation retraction when $G$ is reductive}
Assume $G$ is a reductive group. We construct a strong deformation retraction from $X^{\an}$ onto $\bfp(X^{\an})$, using the same method as the proof of \cite[Theorem 6.2.1]{BerkBook}. Fix a maximal torus $T$ in $G$, Borel subgroup $B_+$ containing $T$, and let $B_-$ be the opposite Borel i.e. the unique Borel subgroup of $G$ such that $B_+\cap B_-=T$. Let $U_+$ be the unipotent radical of $B_+$, and let $U_-$ be the unipotent radical of $B_-$. Then the map:
\[
U_-\times T\times U_+\rightarrow G
\]
given by the multiplication map, is in fact an open immersion \cite[Exposé XXII]{SGA3}. Denote the image by $\Omega$. The groups $U_-$ and $U_+$ are isomorphic affine $N$-space for some $N$. Let $U_+=\spec k[u_1,\ldots u_N]$, and $U_-=\spec k[u_{-1},\ldots u_{-N}]$, and let $T=\spec (k[t_1^{\pm1},\ldots t_n^{\pm 1}) ]$. Then the coordinate ring of $\Omega$, $k[\Omega]$, is given by $\spec k[t_i^{\pm 1},u_j]$. Then set $T_i=t_i-1$. For each $\lambda\in [0,1]$ there is a point $\bfg_{\lambda}\in k[\Omega]^{\beth}$ given by the semivaluation $\gnor_{\bfg_{\lambda}}$, which is defined by:
\[
\abs{\sum_{I,J} \alpha_{I,J}T^I u^J}_{\bfg_{\lambda}}=\min_{I,J}\abs{\alpha_{I,J}}\lambda^{|I|+|J|}
\]
where $I$ and $J$ denote multi-indices.Then $\bfg_1=\bfg$ and $\bfg_0=e$, the identity of $G$. Thus we have a map of sets:
\[
H\colon [0,1]\times X^{\an}\rightarrow X^{\an}
\]
 $(\lambda,p)\mapsto \bfg_{\lambda}\ast p$. Furthermore $H(0,p)=p$ and $H(1,p)=\bfp(X^{\an})$. Notice that because $\bfg_{\lambda}\in X^{\beth}$ for any $\lambda \in [0,1]$, we can restrict $H$ to a map $[0,1]\times X^{\beth}\rightarrow X^{\beth}$, with $H(0,p)=p$ and $H(1,p)=\bfp(p)$.
\begin{prop}\label{def retract prop}
The map $H$ is a homotopy, and in fact a strong deformation retraction of $X^{\an}$ onto $\bfp(X^{\an})$. Similarly $H$ defines a strong deformation retraction from $X^{\beth}$ onto $\bfp(X^{\beth})$.
\end{prop}

\begin{proof}
That $H$ is continuous follows directly from \cite[Corollary 6.1.2]{BerkBook}. It follows from Proposition \ref{cd prop} part 2 that $H$ is strong deformation retract because we have that $\bfg_{\lambda}\ast \bfg_1=\bfg_1$. Indeed, for $f\in k[\Omega]$ we will show that $\abs{m^{\ast}f}_{\bfg_{\lambda},\bfg_1}=1$. Let $m^{\ast}f=\sum a_{I,J}^{L,M}t^I u^J\otimes t^L u^M$. Without loss of generality we can assume that $a_{I,J}^{L,M}=0$ whenever $I$ or $L$ is less than 0, because otherwise we can factor out a power of $t^{R}\otimes t^S$, and $|\!|t^{R}\otimes t^S|\!|_{\bfg_{\lambda},\bfg_1}=1$. But then we can rewrite the sum by substituting $T_i$ for $t_i-1$, and we have that $m^{\ast}f=\sum b_{I,J}^{L,M}T^I u^J\otimes T^L u^M$. Note that $T^I$ and $u^J$ both evaluate to 0 on the identity. If $f\neq 0$ then $e^{\ast}\times \text{id}\circ m^{\ast}f\neq 0$, so there must be some term $b_{0,0}^{L,M}\neq 0$. But then:
\[
\abs{m^{\ast}f}_{\bfg_{\lambda},\bfg_1}=\abs{\sum b_{I,J}^{L,M}T^I u^J\otimes T^L u^M}_{\bfg_{\lambda},\bfg_1}=\max_{I,J,L,M}(\abs{b_{I,J}^{L,M}})(\lambda^{|I|+|J|})=1
\]
as the maximum is achieved when $|I|+|J|=0$.

\end{proof}

\section{The Luna-Vust Classification of Spherical varieties}\label{Colored fans sec}
Here we review the classification of embeddings of spherical homogeneous spaces via colored fans. The classification was first given by Luna and Vust in \cite{LV}, though we will follow the exposition of \cite{Knop}. Much of the notation for spherical varieties is different than that of toric varieties, we will attempt to align our notation with that of toric varieties, similar to the notation in \cite{Nash1}.

Let $k$ be an algebraically closed and trivially valued field\footnote{In the literature on spherical varieties one does not generally specify a valuation on $k$, but it is assumed the valuations on $k(X)$ are trivial when restricted to $k$.} and let $G$ be a reductive group over $k$. Let $H$ be a closed subgroup such that $G/H$ has an open $B$-orbit, for some Borel subgroup $B\subseteq G$. A \textbf{$G/H$-embedding} is a normal $G$-variety $X$ with a $G$-equivariant open embedding $G/H\hookrightarrow X$. Note that if $X$ is a spherical variety with open $B$-orbit $B_0$, then it is a $G_0$-embedding where $G_0$ is the open $G$-orbit $GB_0$. The Luna-Vust theory is a generalization of the toric case which puts embeddings of a torus in bijection with fans. Here $G/H$-embeddings are put in bijection with \textit{colored fans}. To define colored fans and state the classification we will need to introduce some data associated to $G/H$, primarily a lattice $M$, its dual lattice $N$, a convex cone $\cV$ in $N\otimes_{\bZ}\bR$, and a finite set of \textit{colors} $\cD(G/H)$. We refer the reader to Example \ref{SL2 ex} for an explicit demonstration of the theory.

\subsection{Data associated to $G/H$}

Let $\fX(B)$ be the character lattice of $B$ and define:
\[
k(G/H)^{(B)}=\{f\in k(G/H)^{\times}\, | \, \text{there exists $\chi_f\in \fX(B)$ such that for any $b\in B(k),$ } b\cdot f=\chi_f(b)f\}.
\]
We say elements of $k(G/H)^{(B)}$ are $B$ semi-invariant. There is a map $k(G/H)^{(B)}\rightarrow \fX(B)$ given by:
\[
f\mapsto \chi_f
\]
and the kernel is exactly $k^{\times}$. Let the image of this map be $M$, set $N=\hom(M, \bZ)$, $M_{\bR}=N\otimes_{\bZ}\bR$, and $N_{\bR}=N \otimes_{\bZ} \bR$. We say $M$ is the \textbf{weight lattice} of $G/H$. Notice that every  valuation $\gnor$ on $k(G/H)$ defines an element $\varrho(\gnor)\in N_{\bR}$ by:
\[
\varrho(\gnor)(\chi_f)=-\log(\abs{f}).
\]
In particular if $\cV$ is the set of $G(k)$-invariant valuations on $k(G/H)$ then the map $\varrho$ restricted to $\cV$ is an injection by \cite[Corollary 1.8]{Knop} and the image is a finitely generated convex cone \cite[Corollary 5.3]{Knop}. In general though the map $\varrho$ is \textit{not} injective: if $\gnor_1$ and $\gnor_2$ are distinct valuations, but not $G(k)$-invariant, then it is possible that $\varrho(\gnor_1)=\varrho(\gnor_2)$. We say $\cV$ is the \textbf{valuation cone} of $G/H$ and identify it with its image in $N_{\bR}$. The cone $\cV$ spans $N_{\bR}$ \cite[Corollary 5.3]{Knop} and is cosimplicial, i.e. there exist linearly independent elements $\chi_1,\ldots \chi_r $ in the dual space $ N_{\bR}^{\ast}$ such that $\cV=\{v\in N_{\bR} \,\, | \,\, \chi_i(v)\geq 0, \text{ for all $i$}  \}$ \cite[Theorem 5.4]{Knop}.

Now let $X$ be a $G/H$-embedding, and let $Y$ be a $G$-orbit. Then define:
\begin{gather*}
    \cD(X):=\{D\, | \, \text{$D$ a $B$-stable prime divisor in $X$}\} \\
    \cD_Y(X):=\{D\, | \, \text{$D$ a $B$-stable prime divisor in $X$ and $D\supseteq Y$}\}
\end{gather*}
We say that the \textbf{colors} of $G/H$ are $\cD(G/H)$. Notice that every prime divisor $D$, in particular every color, defines a valuation on $k(G/H)$, which we denote $\gnor_D$. These valuations then define an element of $N_{\bR}$, $\varrho(\gnor_D)$, so in particular each color defines an element of $N_{\bR}$. For convenience we may write $\varrho(D)$ when we mean $\varrho(\gnor_D)$. 

When we say `the data associated to $G/H$' we mean the vector space $N_{\bR}$, the valuation cone $\cV$ in $N_{\bR}$, and the set of colors $\cD(G/H)$.

\subsection{Colored fans}

We can now define colored fans and colored cones. By a rational point of $N_{\bR}$ we mean a point in $N\otimes_{\bZ}\bQ$. 

\begin{dfn}
A \textbf{colored cone} is a pair $(\sigma,\cF)$ where $\sigma$ is a strictly convex cone in $N_{\bR}$, and $\cF$ is a subset of $\cD(G/H)$, such that the following are satisfied:
\begin{enumerate}[label=\textbf{CC\arabic*}, leftmargin=2cm, topsep=0cm, itemsep=0cm]
    \item $\sigma$ is generated by $\varrho(\cF)$ and finitely many rational points of $\cV$
    \item The relative interior of $\sigma$ intersects $\cV$
    \item $0\notin \varrho(\cF)$.
\end{enumerate}
We say $(\sigma',\cF')$ is a \textbf{colored face} of a colored cone $(\sigma,\cF)$ if $\sigma'$ is a face of the cone $\sigma$,  the relative interior of $\sigma'$ intersects $\cV$, and $\cF'=\cF\cap\ \varrho^{-1}(\sigma)$. A \textbf{colored fan}, $\fF$, is a finite set of colored cones containing all the colored faces of any colored cone in $\fF$, and for any $\gnor \in \cV$ there is at most one colored cone $(\sigma,\cF)\in\fF$ with $\gnor$ in the relative interior of $\sigma$.
\end{dfn}

We associate a colored cone to each $G$-orbit in $X$, $Y$, as follows. First define:
\begin{gather*}
     \cB_Y(X)=\{D \in \cD_Y(X)\, | \, D \text{ is $G$-stable}\} \\
    \cF_Y(X)=\{D \in\cD_Y(X)\, | \,  D \text{ is not $G$-stable}\}
\end{gather*}
and let $\sigma_Y(X)$ be the cone in $N_{\bR}$ generated by $\rho(\cB_Y(X))$ and $\rho(\cF_Y(X))$.  We identify the elements of $\cF_Y(X)$ with elements of $\cD(G/H)$ by intersecting each element $D$ with $G/H$, as $D\cap G/H$ is dense in $D$ when $D$ isn't $G$-stable. Then we get a colored cone $(\sigma_Y(X),\cF_Y(X))$. Just as affine toric varieties correspond to cones, we have a particular type of $G/H$-embedding corresponding to colored cones. A $G/H$-embedding is \textbf{simple} if it has a unique closed $G$-orbit. 

\begin{thm}\cite[Theorem 3.1]{Knop}
The map 
\[
X\mapsto (\sigma_Y(X),\cF_Y(X))
\]
gives a bijection between isomorphism classes of simple $G/H$-embeddings and colored cones.
\end{thm}

We can construct all simple subembeddings of $X$ in the following way. Let $Y$ be a $G$-orbit in $X$, then define:
\begin{equation}\label{bestopen Eqn}
X_0:= X \setminus\left( \bigcup_{D\in\cD(X)\setminus \cD_Y(X)} D \right).
\end{equation}
Note that $\cD(X)\setminus \cD_Y(X)$ consists of all divisors $D$ such that $\gnor_D$ is \textit{not} an element of $\sigma_Y(X)$.

\begin{thm}\cite[Theorem 2.1]{Knop}\label{fundopen Thm}
The open subvariety $X_0$ is a  $B$-stable, affine open subset of $X$. The intersection $Y\cap X_0$ is an open $B$-orbit, and $GX_0$ is a simple embedding with closed orbit $Y$.  
\end{thm}

When $X$ is simple we have that $GX_0=X$. The following theorem describes the cone $\sigma_Y(X)$ in this case.

\begin{thm}\cite[Theorem 2.5]{Knop}\label{coloredconevals Thm}
Let $X$ be a simple embedding with closed orbit $Y$, and let $X_0$ be as in \ref{bestopen Eqn}.
\begin{enumerate}
    \item $k[X_0]^{(B)}=\{f\in k(G/H)^{(B)}\, | \, \chi_f\in (\sigma_Y(X))^{\vee}\}$.
    \item The center of $\gnor$ exists if and only if $\gnor\in \sigma_Y(X)$.
    \item The center of $\gnor$ is $Y$ if and only if $\gnor$ is in the relative interior of $\sigma_Y(X)$.
\end{enumerate}
\end{thm}

Furthermore we have the following about the faces of $\sigma_Y(X)$ and the orbits of $X$.

\begin{lem}\cite[Lemma 3.2]{Knop}\label{orbitcone Lem}
The map
\[
Z\mapsto (\sigma_Z(X),\cF_Z(X))
\]
gives a bijection between orbits $Z$ whose closure contains $Y$ and the colored faces of $(\sigma_Y(X),\cF_Y(X))$.
\end{lem}

Due to the above, and the fact that there are finitely many $G$-orbits in $X$ (\cite{Knop}), we have that the set $\fF(X):=\{(\sigma_Y(X),\cF_Y(X))\, | \, \text{$Y$ is a $G$-orbit}\}$ is a colored fan. We can now state the main theorem in the classification.

\begin{thm}\cite[Theorem 3.3]{Knop}\label{colorfan Thm}
The map:
\[
X\mapsto \fF(X)
\]
is a bijection between isomorphism classes of $G/H$-embeddings and colored fans.
\end{thm}

\begin{ex}
The classification of toric varieties by fans is a special case of the Luna-Vust theory. Let $G=T$ be a torus, let $H$ be trivial. Then $\cD(T)=\emptyset$, $N$ is the cocharacter lattice of $T$, and $T$-embeddings are classified by fans in $N_{\bR}$.
\end{ex}

\begin{ex}\label{SL2 ex}
Let $G=\text{SL}_2$ and let $H$ be the subgroup of upper triangular matrices with $1$s on the diagonal. Let $O$ denote the origin of $\bA^{2}$, there is an isomorphism $G/H\rightarrow \bA^{2}_k\setminus O$ given by mapping a matrix to its first column. The complement of the $x$-axis is an open orbit of the Borel subgroup of upper triangular matrices, $B$. We have that $\fX(B)\cong \bZ$, where the generator is the character $\chi$ given by:
\[
\left[ {\begin{array}{cc}
    a & b \\
    0 & a^{-1} \\
  \end{array} } \right] \mapsto a.
\] Furthermore $k(G/H)^{(B)}/k^{\times}=\{y^{n}\, | \, n\in \bZ\}$ and $\chi_{y^n}=\chi^{-n}$, so $M=\fX(B)$. Thus $N_{\bR}\cong \bR$. Let $D$ be the $x$-axis in $\bA^{2}$. Consider the embedding $\bA^{2}\hookrightarrow \bP^2$ given by $(a,c)\mapsto [a:c:1]$. Extend the $\text{SL}_2$-action to $\bP^2$ by letting $\text{SL}_2$ act on the first two homogeneous coordinates while preserving the third.

Let $L$ be the line at infinity in $\bP^2$. Then both $\gnor_D$ and $\gnor_L$ are $G(k)$-invariant and furthermore $\varrho(\gnor_D)=1$ and $\varrho(\gnor_{L})=-1$, thus $\cV=N$. There is one color, given by the $x$-axis. The table below consists of a complete list of all $G/H$-embeddings. The hollow circle denotes the color $D$, and $E$ denotes the exceptional divisor in the blow ups at $O$: $\text{Bl}_0\bA^2$ and $\text{Bl}_O\bP^2$.
\begin{center}
\begin{tikzpicture}
	 \draw [thick] (0,7) 
      to  (8,7)
      to (8,2.5) 
      to  (0,2.5)
      to (0,7);
      \draw [thick] (2,7) 
      to  (2,2.5);
      \draw [thick] (4,7) 
      to  (4,2.5);
      \draw  (0,6) 
      to  (8,6);
      \draw (1,6.5) node [black] {$X$};
      \draw (3,6.75) node [black] {Closed };
      \draw (3,6.25) node [black] {$G$-orbits};
      \draw (6,6.5) node [black] {$\fF(X)$};
      \draw (1,5.5) node [black] {$\bA^{2}\setminus O$};
      \draw (1,5) node [black] {$\text{Bl}_O \bA^2$};
      \draw (1,4.5) node [black] {$\bP^2\setminus O$};
      \draw (1,4) node [black] {$\text{Bl}_O \bP^2$};
      \draw (1,3.5) node [black] {$\bA^{2}$};
      \draw (1,3) node [black] {$\bP^2$};
      \draw (3,5) node [black] {$E$};
      \draw (3,4.5) node [black] {$L$};
      \draw (3,4) node [black] {$L,\, E$};
      \draw (3,3.5) node [black] {$O$};
      \draw (3,3) node [black] {$L, \, O$};
      \draw [black,fill] (6,5.5) circle [radius=0.07];
      \draw [black,fill] (6,5) circle [radius=0.07];
      \draw [arrow] (6,5) -- (7.5,5);
      \draw [black,fill] (6,4.5) circle [radius=0.07];
      \draw [arrow] (6,4.5) -- (4.5,4.5);
      \draw [black,fill] (6,4) circle [radius=0.07];
      \draw [arrow] (6,4) -- (7.5,4);
      \draw [arrow] (6,4) -- (4.5,4);
      \draw [black,fill] (6,3.5) circle [radius=0.07];
      \draw [arrow] (6,3.5) -- (7.5,3.5);
      \draw [VioletBlue,thick] (6.5,3.5) circle [radius=0.1];
      \draw [black,fill] (6,3) circle [radius=0.07];
      \draw [arrow] (6,3) -- (7.5,3);
      \draw [VioletBlue,thick] (6.5,3) circle [radius=0.1];
      \draw [arrow] (6,3) -- (4.5,3);
      \draw [dashed] (4.5,5.5) -- (7.5,5.5);
      \draw [dashed] (4.5,5) -- (7.5,5);
      \draw [dashed] (4.5,4.5) -- (7.5,4.5);
      \draw [dashed] (4.5,4) -- (7.5,4);
      \draw [dashed] (4.5,3.5) -- (7.5,3.5);
      \draw [dashed] (4.5,3) -- (7.5,3);
\end{tikzpicture}
\end{center}
\end{ex}

\subsection{Orbits and Orbit Closures}
Let $Z$ be the $G$-orbit of a closed point in $X$. Notice that from Theorem \ref{fundopen Thm} it follows that $Z$ is spherical. Let $\overline{Z}$ be the closure of $Z$ in $X$, so $\overline{Z}$ is a $Z$-embedding. Let $M(Z)$ be the weight lattice of $Z$, $N_{\bR}(Z)=\hom(M(Z),\bR)$, $\cV(Z)$ be the valuation cone, and $\cD(Z)$ be the $B$-stable prime divisors. We will relate the above to $M$, $N_{\bR}$, $\cV$, and $\cD(G/H)$. Furthermore we will give the relationship between the colored fan of $X$ and that of $\overline{Z}$. See \cite[Section 4]{Knop} for further discussion of the following.

Given $z\in Z$ there is a map $\varphi \colon  G/H\rightarrow Z$ given by $gH\mapsto gz$. This induces a map of function fields $k(Z)\hookrightarrow k(G/H)$. The map $k(Z)\hookrightarrow k(G/H)$ gives an inclusion $M(Z)\hookrightarrow M$ and after dualizing, a surjection $\varphi_* \colon  N_{\bR}\rightarrow N_{\bR}(Z)$. Let $\tau=\sigma_{Z}(X)$, $N(\tau)=N_{\bR}/\spn(\tau)$, and let $M(\tau)$ be the dual vector space to $N(\tau)$. Consider the following subset of $\cD$:
\[
 \cF_{\varphi}=\{D\in \cD \, | \, \text{$D$ maps dominantly to $Z$}\}.
\]
 Let $\cF_{\phi}^{\text{c}}=\cD\setminus \cF_{\phi}$.

\begin{thm}\cite[Theorem 4.4]{Knop}\label{coloredsubspace Thm}
We have that $N(\tau)=N(Z)$, the map $\varphi_*$ is the projection $N\rightarrow N(\tau)$, $M(\tau)=M(Z)$, $\varphi_*(\cV)=\cV(Z)$, and there is a bijection $\cF_{\varphi}^{\text{c}}\rightarrow \cD(Z)$ given by $\varphi$.
\end{thm}

To state the relationship between $\fF(X)$ and $\fF(\overline{Z})$ define the following colored fan:
\[
\text{Star}(\tau):=\{(\varphi_*(\sigma),\varphi_*(\cF\cap \cF_{\varphi}^{\text{c}}))\, | \, (\sigma,\cF)\in \fF(X) \text{ and $(\tau,\cF_Z(X))$ is a colored face of $(\sigma,\cF)$.}\}.
\] Then as a result of \cite[Theorem 4.5]{Knop} we have that. 

\begin{thm}\label{godforsakenlemma2 Thm}
Let $X$ be a $G/H$-embedding, and let $Z\subseteq X$ be a $G$-orbit. Let the Zariski closure of $Z$ in $X$ be $\overline{Z}$. Then $\text{Star}(\sigma_Z(X))=\fF(\overline{Z})$.
\end{thm}

\section{Spherical Tropicalization}\label{tropicalization section}

In this Section we review tropicalization for toric varieties, tropicalization of spherical varieties, and prove Theorem \ref{IntroAthm}. We also discuss applications of Theorem \ref{IntroAthm}.  Let $k$ be an algebraically closed and trivially valued field, let $G$ be a reductive group over $k$, let $X$ be a spherical $G$-variety with open $G$-orbit $G/H$.

\subsection{Tropicalization of toric varieties}

 Let $\overline{K}=\ps$ be the field of Puiseux series, equipped with the $u$-adic valuation which we denote $\gnor_u$. Let $T=\spec k[t_1^{\pm 1},\ldots t_n^{\pm 1}]$ and let be $W$ a subvariety of $T$. There is a map of sets $\trp_T\colon  W(\overline{K})\rightarrow \bR^n$ given by:
\[
(t_1,\ldots t_n)\mapsto (-\log(\abs{t_1}_u),\ldots -\log(\abs{t_n}_u)).
\]
We say the closure of $\trp_T(W(\overline{K}))$ is the \textbf{tropicalization of} $W$, denoted $\trop_T(W)$.

This construction extends to subvarieties of toric varieties, as follows. Let $X$ be a toric variety, with dense torus $T$. Let $N$ be the cocharacter lattice of $T$, and let $F(X)$ be the fan in $N_{\bR}=N\otimes_{\bZ}\bR$ associated to $X$. Given a cone $\sigma \in F(X)$, consider the vector space $N(\sigma)=N_{\bR} / \spn (\sigma) $. As a set we define we the tropicalization of $X$ to be:
\[
\trop_T(X):=\bigsqcup_{\sigma\in F(X)} N(\sigma).
\]

To define a topology on $\trop_T(X)$ we first consider the case when $X$ is affine; so $X$ corresponds to a cone $\sigma \in F(X)$. Let $M_{\bR}$ be the dual vector space to $N_{\bR}$; let $\sigma^{\vee}\subseteq M_{\bR}$ denote linear functions on $N_{\bR}$ that are nonnegative on $\sigma$. Let $\overline{\bR}=\bR\cup \{\infty \}$.  If $\hom(\sigma^{\vee}, \overline{\bR})$ denotes semigroup homomorphisms $\sigma^{\vee}\rightarrow \overline{\bR}$, then there is a natural bijection $\trop_T(X)\rightarrow\hom(\sigma^{\vee}, \overline{\bR})$ defined as follows. The fan $F(X)$ consists of the faces of $\sigma$. Let $\tau$ be a face of $\sigma$ and let $M(\tau)$ be the dual of $N(\tau)$. Recall that $M(\tau)=\tau^{\perp}\subseteq M_{\bR}$. So $\alpha\in N(\tau)$ defines a semigroup homomorphism $\tilde{\alpha}\colon \sigma^{\vee}\rightarrow \overline{\bR}$
 \begin{displaymath}
\tilde{\alpha}(v)= \left\{
         \begin{array}{cc}
       \alpha(v), & v \in \tau^{\perp}\cap\sigma^{\vee} \\
       \infty, & \text{else} \\
         \end{array}
   \right.
   \end{displaymath}
This then defines a bijection $\trop_T(X)\rightarrow\hom(\sigma^{\vee}, \overline{\bR})$ and we give $\trop_T(X)$ the topology $\hom(\sigma^{\vee}, \overline{\bR})$ inherits from $\overline{\bR}^{\sigma^{\vee}}$. For a general toric variety $X$ we glue the tropicalizations of affine toric subvareities according to the shared faces of their respective cones in $F(X)$. For further discussion of tropicalization of toric varieties and their subvarieties see \cite{Analytification}. 

Recall that $X$ is a disjoint union of tori which are in bijection with the cones of $F(X)$. Thus there is a tropicalization map $X(\overline{K})\rightarrow \trop_T(X)$ which is defined by taking a disjoint union of the tropicalization maps of the tori $T(\sigma)\rightarrow N(\sigma) $, where $T(\sigma)$ is the torus in $X$ corresponding to $\sigma \in F(X)$. Thus for any subvariety $W$ of a toric variety $X$ we have a tropicalization map from $W(\overline{K})$ into $\trop_T(X)$. 

\subsection{Tropicalization of spherical varieties}

 In \cite{TV} Tevelev and Vogiannou define a tropicalization map $\trp_G\colon G/H(\overline{K})\rightarrow \cV$. The map is given as follows. For each $x\in G/H(\overline{K})$ define $\gnor_{x}$ to be the valuation given by:
\[
\abs{f}_x=\abs{gf(x)}_u \quad g\in U_f
\]
as in equation \ref{TheVal Eqn}. Then define $\trp_G(x)=\gnor_{x}$, this gives a map $G/H(\overline{K})\rightarrow \cV$. For a subvariety $W\subseteq G/H$ the \textbf{tropicalization} is the closure of $\trp_G(W(\overline{K}))$, in particular the tropicalization of $G/H$ is $\cV$.

In \cite{Nash1} the above is extended to construct a tropicalization for any $G/H$-embedding $X$, this is done similarly to the toric case. As a set the \textbf{tropicalization}, denoted $\trop_G(X)$, is defined to be:
\[
\trop_G(X):= \bigsqcup_{(\tau,\cF)\in \fF(X)}\cV(\tau)
\]
where $\cV(\tau)$ is the valuation cone of the $G$-orbit associated to $\tau$.

\begin{rem}
If $G$ is a torus and $X$ is a toric variety then the cones $\cV(\tau)$ are exactly the vector spaces $N(\tau)$ as in the previous subsection.
\end{rem}

We now put a topology on $\trop_G(X)$. Consider first the case when $X$ is simple with closed orbit $Y$; so $X$ corresponds to the colored cone $(\sigma_Y(X),\cF_Y(X))$. Let $\sigma= \sigma_Y(X)$. Recall that $\hom(\sigma^{\vee},\overline{\bR})$ is in bijection with the disjoint union of the vector spaces $N(\tau)=N_{\bR}/\spn(\tau)$, where $\tau$ ranges over the faces of $\sigma$. The bijeciton is given by $\alpha \mapsto \tilde{\alpha}$ where
\begin{displaymath}
\tilde{\alpha}(v)= \left\{
         \begin{array}{cc}
       \alpha(v), & v \in \tau^{\perp}\cap\sigma^{\vee} \\
       \infty, & \text{else} \\
         \end{array}
   \right.
   \end{displaymath}
The colored faces of $(\sigma,\cF_Y(X))$ are exactly the faces of $\sigma$ whose relative interior intersects $\cV$. Given such a face $\tau$, $N(\tau)$ contains the valuation cone $\cV(\tau)$. Thus we have that $\trop_G(X)\subseteq \hom(\sigma^{\vee},\overline{\bR})\subseteq \overline{\bR}^{\sigma^{\vee}}$ and we give $\trop_G(X)$ the induced topology. For $X$ not necessarily simple, $\trop_G(X)$ is obtained by tropicalizing the simple spherical subvarieties and gluing the valuation cones of shared orbits.

\begin{rem}
Let $X$ be simple with closed orbit $Y$, and $\sigma=\sigma_Y(X)$. In general $\trop_G(X)$ is not equal to $\hom(\sigma^{\vee},\overline{\bR})$ as we do not necessarily have that $\sigma\subseteq \cV$ or that $\cV(\tau)=N(\tau)$, for a given colored face $(\tau,\cF)$. 
\end{rem}

\begin{ex}
Let $X=\text{Bl}_O\bP^2$ be the embedding of $\text{SL}_2/H$ as described in example \ref{SL2 ex}. Then the tropicalization is given by
\begin{center}
\begin{tikzpicture}
     \draw [black,fill] (6,3) circle [radius=0.07];
      \draw  (6,3) -- (3,3);
       \draw [black,fill] (3,3) circle [radius=0.07];
\end{tikzpicture}
\end{center}
The left most point representing the valuation cone of $L$, the right most of $E$, and the interior representing the valuation cone of $\text{SL}_2/H$. One can find further illustrated examples of this in Sections 4 and 5 of \cite{Nash1}.
\end{ex}

We have a tropicalization map $\trop_G\colon X(\overline{K})\rightarrow \trop_G(X)$ extending the one defined in \cite{TV}. This is because all $G$-orbits of $X$ are spherical homogeneous spaces so we can take the disjoint union of the tropicalization maps. Explicitly, for $x\in X(\overline{K})$ if $x$ is in the $G$-orbit $Y$ then set $\trop_G(x)$ to be the valuation $\gnor_{x}$ on $k(Y)$. Then for a subvariety $W\subseteq X$ we can again define $\trop_G(W)\colon =\trp_G(W(\overline{K}))$

\begin{rem}
In \cite{Nash1} Nash actually defines a \textit{colored tropicalization}. This consists of attaching the finite set of colors associated to a colored face, to the corresponding valuation cone. However we do not deal with the colors here so this was not discussed.
\end{rem}

We now prove Theorem \ref{IntroAthm}.  Recall that in Section 3 we showed that we have a retraction of topological spaces $\bfp\colon X^{\an}\rightarrow X^{\an}$, given by $\ast$-multiplication by the unique point of the Shilov boundary of $G^{\beth}$ (Corollary \ref{Retraction cor}).

\begin{tthmm}\label{Sec5thmA thm}
Let $X$ be a spherical $G$-variety over an algebraically closed and trivially valued field $k$, and let $\overline{K}$ denote Puiseux series in $k$. Then there is a canonical homeomorphism $\iota:\trop_G(X)\rightarrow \bfp(X^{\an})$, and $\bfp\colon X^{\an}\rightarrow \trop_G(X)$ is a retraction of topological spaces such that the following diagram commutes:
\[
\begin{tikzcd}
X(\overline{K}) \arrow[r] \arrow[dr, "\trp_G"'] & X^{\text{an}} \arrow[d, "\bfp"]  \\
 \, &  \trop_G(X).
 \end{tikzcd}
\]
\end{tthmm}

\begin{proof}
First observe that as sets we have an inclusion $\iota:\trop_G(X)\hookrightarrow X^{\an}$ given as follows. For each orbit $Z$ of $X$, and $\gnor\in \cV(Z)$, $\gnor$ defines an element of $X^{\an}$ by:
\begin{equation*}
\gnor \mapsto (\eta_Z,\gnor)
\end{equation*}
where $\eta_Z$ is the generic point of $Z$. All the $G$-invariant subvarieties of a spherical variety are closures of $G$-orbits (\cite{Knop}); thus by Proposition \ref{infinite Prop} we have that $\bfp(X^{\an})=\iota(\trop_G(X))$ as sets. That the diagram commutes as maps of sets follows from Corollary \ref{orbit valuations cor}. Finally, $\bfp:X^{\an}\rightarrow \bfp(X^{\an})$ is continuous by Proposition \ref{cd prop}. So it remains to show that the topology on $\trop_G(X)$ agrees with the topology of $\bfp(X^{\an})$, i.e. that $\iota$ is a homeomorphism onto its image. We will need to temporarily emphasize the difference between $\trop_G(X)$ and its image in $X^{\an}$. Let  $\cI$ be the set $\iota(\trop_G(X))$, and give $\cI\subseteq X^{\an}$ the subspace topology.

Consider the case when $X$ is simple with closed orbit $Y$. Let $X_0$ be the $B$-stable affine open defined by equation \ref{bestopen Eqn}. The open $X_0$ meets every $G$-orbit in $X$ by Theorem \ref{fundopen Thm}. Thus $X_0^{\an}$ contains $\cI$. The topology on $X_0^{\an}$ is generated by the functions $\ev_f\colon k[X_0]^{\an}\rightarrow \overline{\bR}$, where $\ev_f(\gnor)= \abs{f}$, $f\in k[X_0]$. The functions $\ev_f$ for $f\in k[X_0]^{(B)}$ define the topology on $\trop_G(X)$ by part 1 of Theorem \ref{coloredconevals Thm}. So the inclusion $\iota\colon \trop_G(X)\rightarrow \cI$ is open. Given that for a general $G/H$-embedding $\trop_G(X)$ is given by gluing the tropicalizations of simple subembeddings we have an open bijection $\iota\colon \trop_G(X)\rightarrow \cI$ for all $X$.

Consider the case then when $X$ is proper, then $X^{\an}$ is compact. We know $\cI$ is a closed subset of $X^{\an}$, and therefore $\cI$ is compact. Thus $\iota\colon \trop_G(X)\rightarrow \cI$ is a homeomorphism onto its image. Then for a general $X$ there exists a proper spherical $G$-variety $\overline{X}$ and an open $G$-equivariant embedding $X\hookrightarrow \overline{X}$ by \cite[Theorem 3]{Sumihiro}. But then we have a commutative diagram:  \[ 
\begin{tikzcd}
X^{\an} \arrow[hookrightarrow,r]& \overline{X}^{\an} \\
\trop_G(X) \arrow[u,hookrightarrow] \arrow[hookrightarrow,r]&  \trop_G(\overline{X}) \arrow[u,hookrightarrow] \\
\end{tikzcd}
    \] 
where the composition $\trop_G(X)\hookrightarrow \trop_G(\overline{X})\hookrightarrow \overline{X}^{\an}$ is a homeomorphism onto its image by the above. Furthermore the map $X^{\an}\hookrightarrow \overline{X}^{\an}$ is an open embedding so the inclusion  $\trop_G(X)\hookrightarrow X^{\an}$ is a homeomorphism onto $\cI$.
\end{proof}

\begin{rem}
In \cite{KM} it is shown that tropicalization factors through $X^{\an}$ when $X$ is affine. It follows from Corollary \ref{orbit valuations cor} that $\bfp$ and the map $X^{\an}\rightarrow \trop_G(X)$ given in Definition 5.18 of \cite{KM} agree.
\end{rem}

\begin{cor}
For $W$ a subvariety of $X$, $\trop_G(W)$ is a strong deformation retract of $W^{\an}$.
\end{cor}

\begin{proof}
We have that $\trop_G(W)=\bfp(W^{\an})$ so this follows from Proposition \ref{def retract prop}.
\end{proof}

\begin{rem}
Notice that as a result of Proposition \ref{cd prop} one has that $\trop_G$ defines a functor from the category of spherical $G$-varieties to the category of topological spaces. One can describe the morphism as follows. If $X\rightarrow X'$ is a morphism of $G$-varieties, then for a $Z\subseteq X$ a $G$-orbit, $Z$ maps dominantly onto a $G$-orbit $Z'\subseteq X'$. Then $\gnor \in \cV(Z)$ is mapped to its restriction to $k(Z')$ under the inclusion $k(Z')\hookrightarrow  k(Z)$. This is the same as the construction of tropicalization of morphisms give in Section 2 of \cite{Nash2}.
\end{rem}

From Theorem \ref{Sec5thmA thm} we can see that tropicalization did not depend on our choice to use Puiseux series. Let $L/k$ be an algebraically closed, valued extension, with nontrivial valuation. Define $\trp_{G,L}\colon X(L)\rightarrow \cV$ to be the map given $x\mapsto \gnor_x$, where $x$ is contained in some $G$-orbit $Y$ and $\gnor_x$ is the valuation on $k(Y)$ defined in \ref{subscript val on orbit eqn}. 

\begin{cor}
The closure of $\trp_{G,L}(W(L))$ in $\trop_G(X)$ is equal to $\bfp(W^{\an})=\trop_G(W)$, where $W^{\an}$ is considered as a subspace of $X^{\an}$ induced by the inclusion $W\subseteq X$.
\end{cor}

\begin{proof}
For a subvariety $V\subseteq X$, $W(L)$ is a dense subset of $W^{\an}$, and we have a commutative diagram
\[
\begin{tikzcd}
X(L) \arrow[r] \arrow[dr, "\trp_{G,L}"] & X^{\text{an}} \arrow[d, "\bfp"]  \\
 \, &  \trop_G(X).
 \end{tikzcd}
\] So the closure of $\trp_G(W(L))$ in $\trop_G(X)$ for any subvariety $W\subseteq X$ must be equal to $\bfp(W^{\an})$.
\end{proof}

In \cite{Nash2} Nash gives a description of tropicalization by embedding $X$ into a toric variety and then applying the usual tropicalization process for toric varieties. See \cite[Section 6]{Nash2} for details. In particular this process is used to prove the following result, for which we can give an alternate proof.

\begin{thm}\cite[Theorem 6.7]{Nash2}
If $W\subseteq G/H$ is a subvariety, and $\overline{W}$ is the closure of $W$ in $X$. Then we have that:
\[
\trop_G(\overline{W})=\overline{\trop_G(W)}
\]
where $\overline{\trop_G(W)}$ is the closure taken in $\trop_G(X)$.
\end{thm}
\begin{proof}
This follows from Theorem \ref{Sec5thmA thm}, as $\bfp$ is a retraction.
\end{proof}

\section{Compactification of the colored fan}\label{Generalizing Thuillier sec}
Recall from subsection \ref{beth spaces subsec} that when working over a trivially valued field $X^{\an}$ has a compact analytic subdomain $X^{\beth}$. The space $X^{\beth}$ is defined as the subspace whose points are pairs $(x,\gnor)\in X^{\an}$, where $\gnor$ has center in $X$. In this Section we study $\bfp(X^{\beth})$ and prove Theorem \ref{IntroBthm}. In the case of toric varieties it is shown in \cite[Section 2]{Thuillier} that $\bfp(X^{\beth})$ is the canonical compactification of the fan associated to $X$. We begin by reviewing the work of Thuillier on the toric case.

Let $X$ be a toric variety over $k$ with dense torus $T$. Let $F(X)$ denote the fan associated to $X$. Let $\sigma$ be a cone in $F(X)$ and $U_{\sigma}$ the $T$-stable open affine associated to $\sigma$. We will assume all the cones of $F(X)$ are strictly convex. In \cite{Thuillier} Thuillier defines a map of topological spaces $X^{\beth}\rightarrow X^{\beth}$; which is equal to the map $\bfp$ given by $\ast$-mutliplication by the unique point in the Shilov boundary of $T^{\beth}$. Let $\sigma^{\vee}$ denote positive linear functionals on $\sigma$. Thuillier shows that $\rho^{-1}(T)\cap \bfp(U_{\sigma}^{\beth})$ is naturally homeomorphic to $\hom(\sigma ^{\vee},[0,\infty))$, where $\hom$ denotes semigroup homomorphisms that preserve scaling by $\bR_{\geq0}$ and $[0,\infty)$ has the additive semigroup structure. The space $\hom(\sigma^{\vee},[0,\infty))$ is naturally isomorphic to $\sigma$, as $\sigma$ is strictly convex. Furthermore $\bfp(U^{\beth})$ is compact, and in fact is canonically homeomorphic to $\hom(\sigma^{\vee},[0,\infty])$, the \textbf{canonical compactification} of $\sigma$. We will denote the canonical compactification of a strictly convex cone by $\overline{\sigma}$. In this way $\bfp(X^{\beth})$ is the canonical compactification of the fan $F(X)$.

\begin{rem}\label{star cond rem}
Let $X$ be a spherical $G$-variety with dense $G$-orbit $G/H$. Observe that $\bfp$ has image consisting only of $G(k)$-invariant valuations, and colored fans may contain points of $N_{\bR}$ lying outside $\cV$. Consider Example \ref{star cond example}. Because of this fact we can give a more concise description of $\bfp(X^{\beth})$ when the following holds:
\begin{equation*}\label{support Eqn}
    \bigcup_{(\sigma,\cF)\in\fF(X)}\sigma\subseteq \cV \tag{$\star$}
\end{equation*}
Horospherical varieties, for example, satisfy this condition. However for a symmetric homogeneous space $G/H$ there will be colors whose image in $N_{\bR}$ lies outside $\cV$ and thus symmetric varieties may not satisfy the above condition (see \cite{VustSymmetric} or \cite[Section 3.4]{Perrin}). In Remark \ref{description when star not satisfied rem} we discuss how to describe $\bfp(X^{\beth})$ when \ref{support Eqn} is not satisfied.
\end{rem}

\begin{ex}\label{star cond example}
    As in \cite[Example 4.4]{Nash1}, let $G=\text{GL}_2\times\text{GL}_2$ act on $\text{GL}_2$ via $(g,h)\cdot x= gxh^{-1}$. Note that $\text{GL}_2\cong G/H$ where $H$ is the diagonal subgroup. Let $B$ be the Borel subgroup given by elements of the form $(b,b')$ where $b$ is an upper triangular matrix and $b'$ is lower triangular; then $B$ has an open orbit in $\text{GL}_2$. Let $X$ be the $\text{GL}_2$-embedding given by $2\times 2$ matrices over $k$, and let $X'$ be the $\text{GL}_2$-embedding given by $2\times 2$ matrices of rank at least one over $k$. The data associated to $\text{GL}_2$ and the colored cones associated to $X$ and $X'$ are depicted in Figure \ref{GL_2 fig}. In particular there is a single color, $D$, given by matrices where the bottom right entry is 0. Neither variety is horospherical \cite[Corollary 6.2]{Knop} though both are symmetric varieties and  \ref{support Eqn} does not hold for $X$ but it does for $X'$.
\end{ex}

\begin{figure}[!ht]
\centering
\scalebox{0.7}{
\begin{tikzpicture}
      \node (A) at ( -2.5, -2.5) {};
        \node (B) at (2.5, 2.5) {};
        \node (C) at ( 2.5, -2.5) {};
        \fill [pink] (A.center) -- (B.center) -- (C.center) -- cycle;
       \draw [dashed]  (-2.5,0) -- (2.5,0);
      \draw [dashed] (0,-2.5) -- (0,2.5);
      \filldraw [VioletBlue] (-1.25,1.25) circle [radius=0.1];
      \draw (2.25,1.25) node [DeepPink] {$\cV$};
\end{tikzpicture}
}
\hspace{1.5cm}
\scalebox{0.7}{
\begin{tikzpicture}
      \node (A) at ( -2.5, -2.5) {};
        \node (B) at (2.5, 2.5) {};
        \node (C) at ( 2.5, -2.5) {};
        \fill [pink] (A.center) -- (B.center) -- (C.center) -- cycle;
       \draw [dashed] (-2.5,0) -- (2.5,0);
      \draw [dashed] (0,-2.5) -- (0,2.5);
      \filldraw [VioletBlue] (-1.25,1.25) circle [radius=0.1];
      \draw [arrow] (0,0) -- (-2.5,2.5);
      \draw [arrow] (0,0) -- (2.5,0);
      \node (D) at ( 0,0) {};
        \node (E) at (2.5, 0) {};
        \node (F) at (2.5,2.5) {};
        \node (G) at (-2.5,2.5) {};
        \fill [lightgray, fill opacity=0.5] (D.center) -- (E.center) -- (F.center) -- (G.center) -- cycle;
\end{tikzpicture}
}
\hspace{1.5cm}
\scalebox{0.7}{
\begin{tikzpicture}
      \node (A) at ( -2.5, -2.5) {};
        \node (B) at (2.5, 2.5) {};
        \node (C) at ( 2.5, -2.5) {};
        \fill [pink] (A.center) -- (B.center) -- (C.center) -- cycle;
       \draw [dashed] (-2.5,0) -- (2.5,0);
      \draw [dashed] (0,-2.5) -- (0,2.5);
      \filldraw [VioletBlue] (-1.25,1.25) circle [radius=0.1];
      \draw [arrow] (0,0) -- (2.5,0);
\end{tikzpicture}
}
\caption{To the left is the vector space $N_{\bR}$, the valuation cone  $\cV$, and the one color $D$, for the homogeneous space $\text{GL}_2$. The blue circle depicts $\varrho(\gnor_D)$, and the shaded triangle depicts $\cV$. In the middle is the colored cone for the $\text{GL}_2$-embedding $X$. To the right is the colored cone for $X'$.} \label{GL_2 fig}
\end{figure}
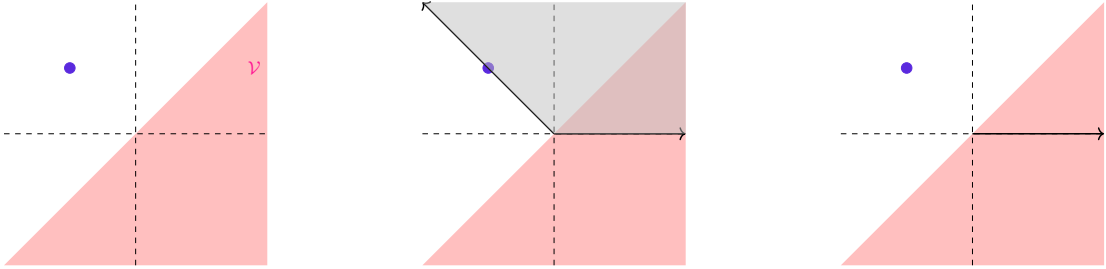

\begin{lem}\label{ConditionDropsToOrbits Lem}
Let $X$ be a $G/H$-embedding satisfying \ref{support Eqn}. Let $Z$ be a $G$-orbit of $X$, and $\overline{Z}$ be the closure of $Z$ in $X$. Then the $Z$-embedding $\overline{Z}$ satisfies \ref{support Eqn}.
\end{lem}

\begin{proof}
Without loss of generality we may assume $X$ is simple. Denote the closed orbit by $Y$. Let $\tau=\sigma_Z(X)$, the face of $\sigma_Y(X)$ corresponding to $Z$. For $z\in Z$ let $\varphi\colon G/H\rightarrow Z$ be the map given by $gH\mapsto gz$, and let $\varphi_*\colon  N_{\bR}\rightarrow N_{\bR}(\tau)$ be the induced map. It follows from Theorem \ref{godforsakenlemma2 Thm} that $\varphi_*(\sigma_Y(X))=\sigma_Y(\overline{Z})$, and we know from Theorem \ref{coloredsubspace Thm} that $\varphi_*(\cV)=\cV(Z)$. So $\sigma_Y(\overline{Z})\subseteq \cV(\tau)$.
\end{proof}

Because colored cones correspond to simple embeddings, we will first consider simple embeddings. Let $X$ be simple with closed $G$-orbit $Y$.  Let $\sigma=\sigma_Y(X)$. We begin by defining a map $\Phi\colon \bfp(X^{\beth})\rightarrow \overline{\sigma}$. The set $\rho^{-1}(G/H)\cap \bfp(X^{\beth})$ consists of $G(k)$-invariant valuations on $k(G/H)$ having center in $X$. Given such a valuation $\gnor$ we know it uniquely defines a semigroup homomorphism $\varrho(\gnor)\colon M\rightarrow \bR$ by:
\[
\varrho(\gnor)(\chi_f)=-\log(\abs{f}).
\]
If $p=(\eta_{G/H},\gnor)\in X^{\beth}$  then denote the restriction of $\varrho(\gnor)$ to $\sigma^{\vee}$ by $\phi_{p}$. By Theorem \ref{coloredconevals Thm} part 2, we have that $\phi_p$  takes values in $[0,\infty)$. Thus we have a map of sets $\rho^{-1}(G/H)\cap \bfp(X^{\beth})\rightarrow \hom(\sigma^{\vee}, [0,\infty))$ given by $p\mapsto \phi_p$.

A general point of $\bfp(X^{\beth})$ is a point $p=(\eta_Z,\gnor)$ where $\eta_Z$ is the generic point of a $G$-orbit $Z$ and $\gnor$ is a $G(k)$-invariant valuation having center in $\overline{Z}$, where $\overline{Z}$ is the Zariski closure of $Z$ in $X$. Let $\tau$ be the face of $\sigma$ corresponding to $Z$. Then $\gnor$ defines an element of $\hom(M(\tau),\bR)=\hom(\tau^{\perp}\cap\sigma^{\vee} ,\bR)$. This extends uniquely to a morphism of semigroups $\phi_{p}\colon \sigma^{\vee}\rightarrow \overline{\bR}$ by setting $\phi_{p}(u)=\infty$ when $u\notin \tau^{\perp}$. But by Theorem \ref{coloredsubspace Thm} and \ref{godforsakenlemma2 Thm} we have that $\tau^{\perp}\cap\sigma^{\vee}=\sigma_Y(\overline{Z})$, thus by Theorem \ref{coloredconevals Thm} part 2 we see that $\phi_{p}$ only takes values in $[0,\infty]$. So we have a map $\Phi\colon \bfp(X^{\beth})\rightarrow \overline{\sigma}$, given by $\Phi(p)= \phi_{p}$.

\begin{prop}\label{mainprop sec 6}
Let $X$ be a simple $G/H$-embedding corresponding to the colored cone $(\sigma,\cF)$, and satisfying \ref{support Eqn}. Then there is a commutative diagram:
 \[ 
\begin{tikzcd}
\bfp(X^{\beth}) \arrow[r,"\sim"]& \overline{\sigma} \\
\bfp(X^{\beth})\cap \rho^{-1}(G/H) \arrow[u,hookrightarrow] \arrow[r,"\sim"]& \sigma \arrow[u,hookrightarrow] \\
\end{tikzcd}
    \] 
    where the horizontal arrows are homeomorphisms and are given by $\Phi$.
\end{prop}

\begin{proof}
Notice that 
\[
\overline{\sigma}=\bigsqcup_{\tau} \hom(\tau^{\perp}\cap \sigma^{\vee},[0,\infty))
\]
where $\tau$ runs over the faces of $\sigma$. But then because $X$ satisfies \ref{support Eqn} we have that:
\[
\bigsqcup_{\tau} \hom(\tau^{\perp}\cap \sigma^{\vee},[0,\infty))= \bigsqcup_{Z} \hom(\sigma_Y(\overline{Z})^{\vee},[0,\infty))
\]
where $Z$ runs over the orbits of $X$, and $\overline{Z}$ is the Zariski closure of $Z$ in $X$. Because $X$ satisfies \ref{support Eqn} all elements of $\hom(\sigma_Y(\overline{Z})^{\vee},[0,\infty))$ are uniquely given by $G(k)$-invariant valuations on $k(Z)$, and by Theorem \ref{coloredconevals Thm} part 2, $\hom(\sigma_Y(\overline{Z})^{\vee},[0,\infty)$ is exactly those $G(k)$-invariant valuations with center in $\overline{Z}$. Thus it follows that $\Phi$ defines a bijection $\bfp(X^{\beth})\rightarrow \overline{\sigma}$ and $\bfp(X^{\beth})\cap \rho^{-1}(G/H)\rightarrow \sigma$.

It remains to show that $\Phi$ is a homeomorphism, and because $\overline{\sigma}$ is compact it suffices to show $\Phi$ is open. Let $X_0$ be the $B$-stable affine open from equation \ref{bestopen Eqn}, note that $X_0$ intersects all $G$-orbits of $X$. So $\bfp(X^{\beth})\subseteq X_0^{\beth}$. Recall from Theorem \ref{coloredconevals Thm} part 1 that $k[X_0]^{(B)}=\{f\in k(G/H)^{(B)}\, | \, \chi_f\in (\sigma_Y(X))^{\vee}\}$. But then evaluation at the functions $f\in k[X_0]^{(B)}$ generates the topology on $\overline{\sigma}$, so $\Phi$ is open.
\end{proof}

As a corollary we have the following.

\begin{tthmm}
There is a strong deformation retraction from $X^{\beth}$ onto $\bfp(X^{\beth})$. If the colored fan $\fF(X)$ satisfies:
\[
\bigcup_{(\sigma,\cF)\in \fF(X)}\sigma \subseteq \cV. \tag{$\star$}
\]
Then $\bfp(X^{\beth})$ is homeomorphic to the the canonical compactification of the fan $F(X)=\{\sigma \, | \, (\sigma,\cF)\in \fF(X)\}$.
\end{tthmm}

\begin{proof}
That $X^{\beth}$ strong deformation retracts onto $\bfp(X^{\beth})$ space follows from Corollary \ref{def retract prop}. By Proposition \ref{mainprop sec 6} if $X$ satisfies \ref{support Eqn} then $\bfp(X^{\beth})$ is homeomorphic to the canonical compactification of the fan $F(X)=\{\sigma \, | \, (\sigma,\cF)\in \fF(X)\}$.
\end{proof}

\begin{rem}\label{description when star not satisfied rem}
Because the map $\bfp$ maps into the closure of $\cV$ in $X^{\beth}$ by construction, one will not have the above for a general spherical variety. This is because a colored cone need not be contained in $\cV$. However given a simple embedding $X$, one will have that $\bfp(X^{\beth})=\overline{\sigma\cap \cV }$, where the closure is taken in $\hom(\sigma^{\vee},[0,\infty])$. Consider for example the $\text{GL}_2$-embedding $\bfp(X^{\beth})$ as in Figure \ref{GL_2 fig}, one can see $\bfp(X^{\beth})$ depicted in Figure \ref{GL_2 compactification fig}.
\end{rem}

\begin{figure}[!ht]
\centering
\begin{tikzpicture}
      \draw  (0,0) -- (-2.5,2.5);
      \draw (0,0) -- (2.5,0);
      \draw (2.5,0) -- (2.5,2.5);
      \draw (-2.5,2.5) -- (2.5,2.5);
      \node (A) at ( 0, 0) {};
        \node (B) at (2.5, 2.5) {};
        \node (C) at ( 2.5, 0) {};
        \fill [pink] (A.center) -- (B.center) -- (C.center) -- cycle;
      \draw [DeepPink] (A.center) -- (B.center) -- (C.center) -- cycle;
            \node (D) at ( 0,0) {};
        \node (E) at (2.5, 0) {};
        \node (F) at (2.5,2.5) {};
        \node (G) at (-2.5,2.5) {};
        \fill [lightgray, fill opacity=0.5] (D.center) -- (E.center) -- (F.center) -- (G.center) -- cycle;
\end{tikzpicture}
\caption{The canonical compactification of the cone $\sigma$ associated to $X$, along with $\bfp(X^{\beth})$ depicted as the pink shaded region to the right.} \label{GL_2 compactification fig}
\end{figure}
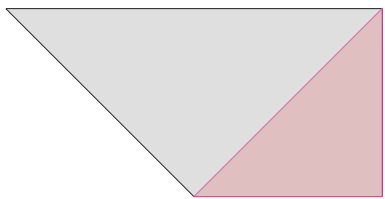

\begin{rem}
Because of this difference between spherical varieties that satisfy \ref{support Eqn} and those that do not, it may be interesting to describe $\ast$-multiplication by $\bfb$, the unique point in the Shilov boundary of $B^{\beth}$. All of the results described in Section \ref{starmult sec} regarding $\bfg$ apply to $\bfb$, including the construction of the homotopy from $X^{\beth}$ to the image of $\ast$-multiplication by $\bfb$.
\end{rem}


\nocite{*}

\renewcommand{\section}[2]{\subsection#1{#2}\vspace{-5pt}} 
\printbibliography

\end{document}